\documentclass[12pt]{amsart}  
\usepackage{amsmath}
\usepackage{amsfonts}  
\usepackage{amssymb}     
\usepackage{graphicx} 
\usepackage{MnSymbol}   

\numberwithin{equation}{section}
 
 \newtheorem{thm}{Theorem}[section]
 \newtheorem{lem}[thm]{Lemma}
 \newtheorem{exam}[thm]{Example}
 \newtheorem{prop}[thm]{Proposition}
 \newtheorem{cor}[thm]{Corollary}
 \newtheorem{rem}[thm]{Remark}
 \newtheorem{defn}[thm]{Definition}

 \def\Id{\mathop{\rm Id}\nolimits}
 \def\re{\mathop{\rm Re\,}\nolimits}

 \def\dom{\mathop{\rm dom}\nolimits}
 \def\ran{\mathop{\rm ran}\nolimits}
 \def\mul{\mathop{\rm mul}\nolimits}

\newcommand{\C}{{\mathbb C}}   
 \newcommand{\R}{{\mathbb R}}           
\newcommand{\N}{{\mathbb N}}        
\newcommand{\K}{{\mathbb K}}        
\newcommand{\Z}{{\mathbb Z}}          

\author{W. Arendt}
\address{Wolfgang Arendt, Institute of Applied Analysis, University of Ulm. Helmholtzstr. 18, D-89069 Ulm (Germany)} 
\email{wolfgang.arendt@uni-ulm.de}

\author{I. Chalendar}
\address{Isabelle Chalendar,  Universit\'e Gustave Eiffel, LAMA, (UMR 8050), UPEM, UPEC, CNRS, F-77454, Marne-la-Vallée (France)}
\email{isabelle.chalendar@univ-eiffel.fr}

\author{B. Moletsane}
\address{Boitumelo Moletsane, University of the Witswatersrand,  Johannesburg (South Africa)}
\email{boitumelo.moletsane@wits.ac.za}
\title[Semigroups generated by multivalued operators] 
{Semigroups generated by multivalued operators and domain convergence for parabolic problems}   
\begin{document}	
\begin{abstract}   
 The following version of the Lumer--Phillips  Theorem \cite{LP61} is proved: a surjective dissipative operator  is $m$-dissipative and invertible. This result remains true if dissipative linear relations (i.e. multivalued operators) are considered.  The main purpose of this article is to study  relations which generate semigroups. We consider $m$-dissipative relations and also the holomorphic estimate for relations. Such relations   are very useful if domain perturbations for the Laplacian are studied.     
\end{abstract}	
\maketitle
    
\section{Introduction}\label{sec:1}
Multivalued operators occur naturally in many subjects of analysis. For example, in the theory of non-linear semigroups they are commonly used (see \cite{Bre73}, \cite[IV.1]{Sho97}); they are a central object in viability theory \cite{ABSP11}, and, last but not least, the theory of selfadjoint multivalued operators and boundary value problems has  been developed to a far extent (see the monograph \cite{BHS20} by Behrndt, Hassi and de Snoo for a comprehensive presentation).   

The terminology used in the literature is not uniform. We follow the monograph \cite{BHS20} and use the  word "(linear) relation", which is nothing else than a subspace of $X\times X$ where $X$ is the underlying Banach space.  Our aim is to study $m$-dissipative relations in the spririt of semigroups. We prove generation results and results on approximation which we apply to domain convergence for parabolic boundary value problems.   

There are  some remarkable typically linear phenomena which 
are worth it to be  considered in the more general situation of relations.  One of them is the closed range theorem. It is well-known for densely defined operators, but we show that it remains true for arbitary closed relations. As a consequence,  the set of all surjective relations is open, an essential fact used in the proof of one of our main results. It concerns the Lumer--Phillips Theorem which we recall here. 
\begin{thm}\label{th:1.1}
Let $A$ be an operator on $X$. The following assertions are equivalent:
\begin{itemize}
	\item[(i)] $A$ generates a contractive $C_0$-semigroup;
	\item[(ii)] \begin{itemize}
		             \item[(a)] $\dom A$ is dense;
		             \item[(b)] $A$ is dissipative and 
		             \item[(c)] $\ran(\lambda -A)=X$ for some $\lambda >0$.  
		             \end{itemize}
 \end{itemize}  
\end{thm} 
One says that the operator $A$ is \emph{$m$-dissipative} if $(b)$ and $(c)$ are satisfied.   This is equivalent  to   $(0,\infty)\subset \rho(A)$ (the \emph{resolvent set} of $A$) and $\|\lambda R(\lambda, A)\|\leq 1$ for all $\lambda>0$.  However, in some instances, it is much easier to prove the range condition $(c)$ for $\lambda=0$ instead of $\lambda>0$; the Dirichlet Laplacian is one such example (see the proof of Theorem~\ref{th:6.1}). And indeed we prove here the following new version of the Lumer--Phillips Theorem. 
\begin{thm}\label{th:1.2}
Let $A$ be an operator (or merely a relation) on $X$. The following assertions are equivalent:
\begin{itemize}
	\item[(i)] $A$ is $m$-dissipative and invertible;
	\item[(ii)] $A$ is dissipative and surjective.  
\end{itemize}
\end{thm}
 It is remarkable that  an operator  (or a relation) $A$ satisfying $(ii)$ is automatically closed. 
 
 One purpose of the article is to study $m$-dissipative relations. On reflexive spaces they always generate a strongly continuous contractive semigroup $T: [0,\infty) \to{\mathcal L}(X)$ where $T(0)$ is a projection. On general Banach spaces a closed relation is $m$-dissipative if and only if it generates a $1$-Lipschitz continuous integrated semigroup, and this is in turn, can be characterized by the well-posedness of an evolutionary problem. This situation is different for holomorphic semigroups (which here may be degenerate). A closed relation generates a bounded holomorphic semigroup, i.e.  a bounded, holomorphic mapping $T:\Sigma_\alpha\to {\mathcal L}(X)$ satisfying $T(z_1+z_2)=T(z_1)T(z_2)$ for all $z_1,z_2\in \Sigma_\alpha$ if and only if $A$ satisfies the usual holomorphic estimate. Here $\Sigma_\alpha$ is an open sector of angle $0<\alpha\leq \frac{\pi}{2}$ and the semigroup may be degenerate (i.e. there may exist $0\neq x\in X$ such that $T(z)x=0$ for all $z\in \Sigma_\alpha$). One of the main points is to characterize   strong convergence of sequences of semigroups (or integrated semigroups) in terms of the resolvents of their generators. The strongest results are obtained in the case of holomorphic semigroups. Here things are even easier to formulate than in the case of operators since the limit needs merely to be a relation  and the limit semigroup may be degenerate. These results are applied to study domain perturbation. If $\Omega_n,\Omega$ are open sets all contained in a large ball $B$, we consider solutions $u_n:(0,\infty)\to C (\overline{B})$ of 
 \[  \dot{u_n}(t)=\Delta u_n (t)\mbox{ on }\Omega_n,\,\, u_n(t)=0\mbox{ on }\overline{B}\setminus \Omega_n\mbox{ and }u_n(0)=u_0\in C(\overline{B})\]
 and show that $u_n(t)\to u(t)$ uniformly on $\overline{B}$ if $\Omega_n\to\Omega$ in a  suitable sense, where $u$ is the corresponding  solution for $\Omega$. The novelty is uniform convergence, which, so far, seems not to be known for the parabolic problem. For the elliptic problems nearly optimal results are known (see \cite{AD07,AD08})    which can be applied here. 
 
 Several results in our study of relations can  be taken over from the case of operators. For those   we just refer to the proofs in the literature. For others, new ideas are needed. The Lumer--Phillips Theorem in the form of Theorem~\ref{th:1.2} seems new even in the case of  operators; the convergence theorems are more delicate, and also the well-posedness results  need new arguments. 
  We expect
  that the abstract results presented here will also be useful for other applications (than the domain convergence put in the focus here) where multivalued operators occur naturally  (for example for the Dirichlet-to-Neumann operator in certain spectral situations or for Robin boundary conditions on bad domains).

The paper is organized as follows. In Section~\ref{sec:2} we consider arbitrary closed (linear) relations, prove the closed-range theorem and show that surjective relations are open. Section~\ref{sec:3} is devoted to the spectral theory of relations; i.e. we consider the resolvent of a relation. Then $m$-dissipative relations are introduced in Section~\ref{sec:4} where also the new version, Theorem~\ref{th:1.2}, of the Lumer--Phillips Theorem is proved. Also maximal dissipative relations are considered  in this section and a convergence result of Trotter--Kato type is proved. 

In Section~\ref{sec:5} we establish generation theorems and extend the convergence result from resolvents to semigroups and  integrated semigroups. Then we study domain approximation for the heat equation with Dirichlet boundary conditions in Section~\ref{sec:6}.  Section~\ref{sec:7} is devoted to relations generating a holomorphic semigroup with a corresponding Trotter--Kato Theorem. Finally we show in Section~\ref{sec:8} that the holomorphic semigroups generated by the Dirichlet Laplacian converge if the domains converge.

\section{Surjective relations}\label{sec:2}
Let $X$ and $Y$ be Banach spaces. A (linear) \emph{relation} from $X$ to $Y$ is a subspace $A$ of $X\times Y$. 
Since we merely consider linear relations here, we will omit "linear" in general.  

A relation $A$ is \emph{closed} whenever it is closed in $X\times Y$ equipped with the product topology.  The closure $\overline{A}$ of $A$ is defined in the space $X\times Y$. 

We define the \emph{domain}, \emph{range}, \emph{kernel} and \emph{multivalued part} of $A$ by 
 \[ \dom A=\{ x\in X: \exists y\in Y ,(x,y)\in A\}\]
 \[ \ran A =\{y\in Y:\exists x\in X, (x,y)\in A \}\]
 \[  \ker A =\{ x\in X:(x,0)\in A\}\]
 \[   \mul A  =\{  y\in Y: (0,y) \in A \}.\]
 We write $y\in Ax$ if $(x,y)\in A$, which is also equivalent to $(x,y+z)\in A$ for all $z\in \mul A$.  \\
  We say that $A$ is an \emph{operator} if $\mul A=\{0\}$.  In this case there exists a linear map $A_0:\dom A\to Y$ whose graph is $A$, i.e. \[A=\{(x,A_0 x):x\in\dom A\},\]
  and obviously $y\in Ax$ means that $y=A_0x$.  \\

  We denote by ${\mathcal L}(X,Y)$ (resp. ${\mathcal L}(X)$) the space of all linear continuous mappings $Q:X\to Y$ (resp. $Q:X\to X$). We will not identify $Q$ with its graph $\{(x,Qx):x\in X\}$. \\
  
  Given a relation $A$, we define the \emph{inverse relation} $A^{-1}$ by 
  \[ A^{-1} =\{  (y,x):(x,y)\in A\}. \]  
 Our first aim is to characterize when the inverse  of a closed relation is associated with a bounded operator.  
 \begin{prop}\label{prop:2.1}
 Let $A\subset X\times Y$ be a closed relation. The following assertions are equivalent:
 \begin{itemize}
 	\item[(i)] $\exists \alpha >0$ such that $\alpha \|x\|\leq \|y\|$ for all $(x,y)\in A$;
 	\item[(ii)] $\ker A =\{ 0\}$  and $\ran A$ is closed;
 	\item[(iii)] $\ran A$ is closed and  there exists  $Q\in \mathcal L (\ran A, X)$ such that $A^{-1} =\{ (y,Qy):y\in \ran A\}$.  
 \end{itemize}
  \end{prop}
\begin{proof}
	$(i)\Rightarrow (ii)$ By (i), $x=0$ whenever $(x,0)\in A$. In other words $\ker A=\{ 0\}$. 
	Let $(y_n)_n\subset \ran A$, $y\in Y$ such that $y=\lim_{n\to\infty} y_n$. Then there exists $(x_n)_n\subset X$ such that $(x_n,y_n)\in A$ for all $n$. Moreover,  by (i), $(x_n)_n$ is a Cauchy sequence.
	 Therefore there exists $x\in X$ such that $\lim_n x_n=x$.    It follows that  $(x,y)=\lim_{n\to\infty} (x_n,y_n)\in \overline{A}=A$ and then $y\in \ran A$, i.e. $\ran A$ is closed. \\
	$(ii)\Rightarrow (iii)$ Since $\ker A =\{0\}$, for each $y\in \ran A$ there exists a unique $Qy\in X$ such that $(Qy,y)\in A$. Then $Q$ is a linear map from $\ran A$ to $X$ and $A=\{(Qy,y):y\in \ran A\}$, i.e. $A^{-1}= \{(y,Qy):y\in \ran A\}$. Since $A$ is closed, the graph of $Q$ is closed and consequently $Q\in {\mathcal L}(\ran A,X)$. \\
	$(iii)\Rightarrow (i)$ Let $(x,y)\in A$. Then $x=Qy$ and then $\|x\|\leq \|Q\| \|y\|$. If $Y\neq \{0\}$, we can choose $\alpha:=\|Q\|^{-1}$. If $Y=\{0\}$,  then $X=\{0\}$ and we may choose an arbitrary positive constant $\alpha$. 
\end{proof}
We call a  relation $A\subset X\times Y$ \emph{invertible} if $A$ is closed, $\ker A=\{0\}$ and $\ran A=Y$. Thus, by Propostion~\ref{prop:2.1}, $A$ is invertible if and only if there exists $Q\in {\mathcal L}(Y,X)$ such that  $A^{-1}=\{(y,Qy):y\in Y\}$. \\

Let $A$ be a relation in $X\times Y$. Then  the \emph{adjoint relation} $A'\subset Y'\times X'$ is defined by 
\[  A'=\{ (y',x'):  \langle x',x\rangle =\langle y',y\rangle \mbox{ for all } (x,y)\in A\}.\]
We identify $(X\times Y)'$ and $X'\times Y'$. 

If $Z$ is a Banach space, for a subspace $V\subset Z$ we let 
\[  V^\perp :=\{  z'\in Z':\langle z',v\rangle =0 \mbox{ for all }v\in V\}. \]
We also define, for $W\subset Z'$, 
\[ {}^\perp W :=\{ z\in Z,\langle w',z\rangle =0 \mbox{ for all }w'\in W\}.\] 
 We can now describe $A'$ as 
 \[ A'=(-A^{-1})^\perp.\]
 Note that $A'$ is a closed subspace in $Y'\times X'$.  It follows from the definitions that 
 \[  (A^{-1})'=(A')^{-1}.\]
 Our aim is to describe the range and kernel by dual expressions. 
 \begin{prop}\label{prop:2.2}
 Let $A\subset X\times Y$ be a closed relation. Then 
 \begin{enumerate}
 \item[(a)] $\ker A'=(\ran A)^\perp$.
 \item[(b)]	$\ker A={}^\perp (\ran A')$.
 \end{enumerate}
 \end{prop}   
\begin{proof}
(a) follows directly from the definitions.\\
(b) "$\subset$" Let $u\in \ker A$, i.e. $(u,0)\in A$. Let $(y',x')\in A'$. Then $\langle y',y\rangle =\langle x',x\rangle $ for all 
$(x,y)\in A$.  Thus $\langle x',u\rangle =0$. \\
"$\supset$" Let  $u\in X$, $u\not\in \ker A$. Then $(u,0)\not\in A$. By the Hahn--Banach  Theorem there exists $(x'_0, y'_0)\in A^\perp$ such that $\langle x'_0,u\rangle \neq 0$. Since $\langle x'_0,x\rangle +\langle y'_0,y\rangle =0$ for all $(x,y)\in A$,  it follows that  $(-y'_0,x'_0)\in A'$. Therefore $x'_0\in \ran A'$ and thus $u\not\in {}^\perp (\ran A')$.  \\
 "$\subset$" Let $(x,0)\in A$, $x'\in \ran A'$. There exists $y'\in Y'$ such that  $(y',x')\in A'$. Thus $\langle x',x\rangle=0$. We have sgown that $x\in {}^\perp(\ran A')$.  
 \end{proof} 
  A remarkable fact is that $\ran A$ is closed if and only if $\ran A'$ is closed. when $A$ is a closed relation.  For the proof  we follow \cite{Bre11} where $A$ is a densely defined operator. 
  
  The basic result is \cite[Theorem 2.16]{Bre11}. 
  \begin{thm}\label{th:2.3}
  	Let $G$ and $L$ be two closed subspaces of a Banach space $E$. The following assertions are equivalent:
  	\begin{itemize}
  		\item[(a)] $G+L$ is closed in $E$;
  		\item[(b)] $G^\perp + L^\perp$ is closed in $E'$. 
  	\end{itemize} 
  \end{thm}  
From this we deduce the following. 
  \begin{thm}\label{th:2.4}[Closed range theorem]
  	Let $A\subset X\times Y$ be a closed relation. The following assertions are equivalent:
  \begin{itemize}
  	\item[(i)]  $\ran A$ is closed;
  	\item[(ii)]  $\ran A'$ is closed. 
  \end{itemize} 	 
  \end{thm}
  \begin{proof}
  	Let $Z=X\times Y$, $Z'=X'\times Y'$,  $G=A$ and $L=X\times \{0\}$. Then $X\times \ran A=G+L$ and $\ran A'\times Y'=G^\perp +L^\perp$ as it is easy to see. So the claim follows from Theorem~\ref{th:2.3}.  
  \end{proof}
  Now we obtain the following characterizations of sujectivity of a closed relation.
  \begin{thm}\label{th:2.5}
  Let $A$ be a closed relation in $X\times Y$. The following assertions are equivalent:
 \begin{itemize}
\item[(i)] $\ran A=Y$;
\item[(ii)] there exists $\alpha>0$ such that $\alpha \|y'\|\leq \|x'\|$ for all $(y',x')\in A'$; 
\item[(iii)] $ker A'=\{ 0\} $ and $\ran A'$ is closed. 
\end{itemize}
  \end{thm}
\begin{proof}
$(ii)\Longleftrightarrow (iii)$ follows from Proposition~\ref{prop:2.1}. \\
$(i)\Rightarrow (iii)$ By Theorem~\ref{th:2.4}, $\ran A'$ is closed. Moreover, by Proposition~\ref{prop:2.2}, 
$ \ker A' =(\ran A)^\perp=Y^\perp =\{0\}$. \\
$(iii)\Rightarrow (i)$  By Proposition~\ref{prop:2.2}, $(\ran A)^\perp=\ker A'=\{0\}$.  Theorem~\ref{th:2.4}  implies that $\ran A$ is closed. Thus $\ran A={ }^\perp ((\ran A)^\perp )=Y$.  
	
\end{proof}
 We also note the following dual version whose proof is omitted since it goes along the same lines
  \begin{thm}\label{th:2.6}
  	Let $A$ be a closed relation. The following assertions are equivalent:
  	\begin{enumerate}
  		\item[(i)]  $\ran A'=X'$;
  		\item[(ii)] $\exists \alpha >0$ such that $\alpha \|x\|\leq \|y\|$ for all $(x,y)\in A$;
  		\item[(iii)] $\ker A=\{0\}$ and $\ran A$ is closed. 
  	\end{enumerate}  
  \end{thm}	  
From Theorem~\ref{th:2.5} we deduce that surjectivity of closed relations is stable with respect to small perturbations by bounded operators. 

If $A\subset X\times Y$ is a closed relation and $B\in {\mathcal L}(X,Y)$, we define the relation $A+B\subset X\times Y$  by
\[ A+B:=\{  (x,y+Bx):(x,y)\in A\}.\]
Then $A+B$ is closed. Moreover, it is easy to see that
\begin{equation}\label{eq:2.1}
(A+B)' =A' + B'=\{   (y',x'+B' y'):(y',x')\in A'\},
\end{equation}	   
where $B'\in {\mathcal L}(Y',X')$ is the adjoint mapping of $B$. 
\begin{cor}\label{cor:2.7}
Let $A$ be a closed relation in $X\times Y$ such that $\ran A=Y$. Then there exists $\alpha >0$ such that $\ran (A+B)=Y$ whenever $B\in {\mathcal L}(X,Y)$, $\|B\|<\alpha$. 
\end{cor}
\begin{proof}
	By Theorem~\ref{th:2.5} there exists $\alpha>0$ such that $\alpha \|y'\|\leq \|x'\|$ for all $(y',x')\in A'$.  Let $B\in {\mathcal L}(X,Y)$ such that $\|B\|<\alpha$. Then $\|B'\|<\alpha$.  For $(y',x'+B'y')\in A'+B'=(A+B)'$ with $(y',x')\in A'$, we have
	\[    \|y'\|\leq \frac{1}{\alpha} \|x'\|\leq \frac{1}{\alpha} \|x'+B'y'\| +\frac{1}{\alpha} \|B'y'\| \leq \frac{1}{\alpha} \|x'+B'y'\| +\frac{\|B'\|}{\alpha} \|y'\|.\]  
	Thus 
	\[  (\alpha -\|B'\|)\|y'\|\leq \|x'+B' y'\|\]
	for all $(y',x'+B'y')\in (A+B)'$.  It follows from Theorem~\ref{th:2.5} that $\ran (A+B)=X$. 
\end{proof}

\section{Resolvents of relations}\label{sec:3}
In this section we introduce some spectral properties of relations on Banach spaces. Some of them are also treated in the monograph \cite[Section 6, 1.2 and 1.9]{BHS20} in the Hilbert space case. Our notations are slightly different in order to fit better with semigroup theory. We give some proofs to be complete. Let $X$ be a Banach space over $\K=\R$ or $\C$ and let $A\subset X\times X$ be a closed (linear) relation.   \\

For $\lambda\in\K$ the relation $\lambda -A$ is defined by 
\[  \lambda -A=\{(x,\lambda x-y):(x,y)\in A\}.\]
Recall that $\lambda -A$ is called invertible if $\ker (\lambda -A)=\{0\}$ and $\ran (\lambda -A)=X$. \\

The \emph{resolvent set} $\rho(A)$ of $A$ is defined as 
\[  \rho (A):=\{ \lambda\in \K :\lambda -A\mbox{ is invertible}\}.\]
By Proposition~\ref{prop:2.1},   for each $\lambda\in \rho(A)$ there exists a unique mapping $R(\lambda, A)
\in{\mathcal  L}(X)$ such that 
\[   (\lambda -A)^{-1}=\{(y,R(\lambda, A)y):y\in X)\}.\]
We call $R(\lambda, A)$ the \emph{resolvent} of $A$ at $\lambda$. 

As in the case where $A$ is an operator, one has the following property 
(see also \cite[Section 1.2]{BHS20}) .   
\begin{prop}\label{prop:3.1}
Let $A$ be a relation. Let $\lambda_0\in \rho(A)$. Then for all $\lambda\in\K$ with 
\[  |\lambda-\lambda_0|\|R(\lambda_0,A)\|<1,\]
one has $\lambda\in\rho(A)$ and 
\begin{equation}\label{eq:3.1}
	R(\lambda ,A)=\sum_{n=0}^\infty (\lambda_0-\lambda)^n R(\lambda_0,A)^{n+1}.
\end{equation}  
Thus $\rho(A)$ is open, $R:\rho(A)\to{\mathcal L}(X)$, $\lambda\mapsto R(\lambda, A)$ is analytic and 
\begin{equation}\label{eq:3.2}
\frac{R^{(n)}(\lambda)}{n!}=(-1)^nR(\lambda,A)^{n+1}\mbox{ for all }\lambda\in\rho(A),n\in\N_0.	
\end{equation}	
Finally, the \emph{resolvent identity} holds
\begin{equation}\label{eq:3.3}
	R(\lambda)-R(\mu)=(\mu-\lambda)R(\lambda)R(\mu)\mbox{ for all }\lambda,\mu\in \rho(A).
\end{equation}
\end{prop}
\begin{proof}
The first part is shown as for operators (see also \cite[Corollary 1.2.7]{BHS20}). Now we show the resolvent identity \eqref{eq:3.3} (see also \cite[Theorem 1.2.6]{BHS20}).  Let $\lambda,\mu\in\rho(A)$ and $y\in X$. Then for $x\in X$, 
\begin{equation}\label{eq:3.4}
x=R(\lambda, A)y\mbox{ if and only if }y\in (\lambda -A)x=\lambda x -Ax.
\end{equation}
Thus $y\in \lambda R(\lambda,A)y-AR(\lambda,A)y$. Adding $(\mu-\lambda)R(\lambda, A)y$ gives 
\[  (\mu-\lambda) R(\lambda, A)y+y\in \mu R(\lambda ,A)y-AR(\lambda,A)y. \] 
Now \eqref{eq:3.4} implies that  
\[ R(\lambda, A)y=R(\mu,A)((\mu-\lambda)R(\lambda,A)y +y),\]
which is \eqref{eq:3.3}. 
\end{proof}    
As a corollary we note the following property which is well-known if $A$ is an operator. 
\begin{cor}\label{cor:3.2}
Let $\lambda_0\in \K$. Assume that there exists a sequence $(\lambda_n)_{n\geq 1}\subset \rho(A)$ such that $\lim_{n\to\infty} \lambda_n=\lambda_0$ and $\sup_{n\geq 1} \|R(\lambda_n,A)\|<\infty$. Then $\lambda_0\in\rho(A)$.   	
\end{cor}	
We also identify the multivalued space of $A$. 
\begin{lem}\label{lem:3.3n}
Let $A\subset X\times X$ be a closed relation and $\lambda\in \rho(A)$. Then for $x\in X$, $(0,x)\in A$ if and only if $R(\lambda, A)x=0$. 
\end{lem}
This is not difficult to see. We also omit the easy proof of the next result.  
\begin{lem}\label{lem:3.3}  
\begin{enumerate}	
\item[(a)] Let $A\subset X\times X$ be a relation and $\lambda_0\in \rho(A)$. Then 
\[  A=\{  (R(\lambda_0,A)u,\lambda_0 R(\lambda_0,A)u-u):u\in X    \}.\]	
\item[(b)] Let $Q\in {\mathcal L}(X)$, $\lambda_0\in \K$. Then $A=\{  (Qu, \lambda_0Q u-u):u\in X\}$ is the unique relation such that $\lambda_0\in\rho (A)$ and  $R(\lambda_0,A)=Q$. 	
\end{enumerate}
\end{lem}
The following result is also given in \cite[Proposition 1.2.9]{BHS20}.
\begin{prop}\label{prop:3.4}
 Let $\Omega$ be a non-empty subset of $\K$ and $R:\Omega\to {\mathcal L}(X)$ a \emph{pseudo-resolvent} (i.e. $R$ satisfies $R(\lambda)-R(\mu)=(\mu-\lambda)R(\lambda) R(\mu)$ for all $\lambda,\mu\in \Omega$).  Then there exists a unique relation $A$ on $X$ such that $\Omega\subset \rho(A)$ and $R(\lambda)=R(\lambda, A)$ for all $\lambda\in \Omega$. 
	\end{prop} 
\begin{proof}
Let $\lambda_0\in \Omega$ and define $A=\{  (R(\lambda_0)u,\lambda_0 R(\lambda_0) u-u):u\in X\}$. Then by Lemma~\ref{lem:3.3}, $\lambda_0\in\rho(A)$ and $R(\lambda_0,A)=R(\lambda_0)$. Let $\mu\in \Omega$. 
We claim that $(\mu-A)^{-1}=\{ (u,R(\mu)u):u\in X\}$. \\
"$\subset$" Note that  	
\[  (\mu -A)^{-1} =\{  ((\mu -\lambda_0)R(\lambda_0)x+x,R(\lambda_0)x):x\in X\}. \]
Let $(u,R(\lambda_0)x)\in (\mu-A)^{-1}$ where $u=R(\lambda_0)(\mu-\lambda_0)x +x$. Then 
\[  R(\mu) u=(\mu-\lambda_0)R(\mu)R(\lambda_0)x+R(\mu)x=R(\lambda_0)x-R(\mu)x +R(\mu) x=R(\lambda_0)x.\]
Thus $(u,R(\lambda_0)x)=(u,R(\mu)u)$. \\
"$\supset$" Let $u\in X$. We want to know that $(u,R(\mu)u)\in (\mu-A)^{-1}$. Define $x:=(\lambda_0-\mu)R(\mu)u +u$. Then 
\[ R(\lambda_0)x =(\lambda_0-\mu)R(\lambda_0)R(\mu)u+R(\lambda_0)u=R(\mu)u.\]
 Then we have 
\begin{eqnarray*}
	(\mu-\lambda_0)R(\lambda_0)x + x & = & (\mu-\lambda_0) R(\mu) u\ +  (\lambda_0 -\mu)R(\mu)u +u=u
 \end{eqnarray*}
Thus 
\[ (u,R(\mu) u)= ((\mu-\lambda_0)R(\lambda_0) x+x,R(\lambda_0)x)\in (\mu-A)^{-1}.\]
\end{proof}
We need the following lemma. 
\begin{lem}\label{lem:3.6new}
Let $A\subset X\times X$ be a closed relation. Let $U\subset \C$ be open, connected and $R:U\to {\mathcal L}(X)$ holomorphic. Assume that there exists an infinite compact set $K\subset U$ such that $K\subset \rho(A)$ and $R(\lambda)=R(\lambda,A)$ for all $\lambda \in K$. Then $U\subset \rho(A)$ and $R(\lambda )=R(\lambda, A)$ for all $\lambda\in U$.     
\end{lem}
\begin{proof}
a)  Let $\lambda\in K$. Then letting $F_1(\mu)=(\mu-\lambda)R(\lambda)R(\mu)$ and $F_2(\mu)=R(\lambda)-R(\mu)$ we obtain two holomorphic functions on $U$ which coincide on $K$. Thus $F_1(\mu)=F_2(\mu)$ for all $\mu\in U$ by the Uniqueness Theorem.
 
b) Now let $\lambda\in U$ and define $F_1,F_2$ as before. Then $F_1(\mu)=F_2(\mu)$ for $\mu\in K$ by a). Hence $F_1(\mu)=F_2(\mu)$ for all $\mu\in U$. Thus $R:U\to{\mathcal L}(X)$ is a pseudo-resolvent. Now the claim follows from Proposition~\ref{prop:3.4}.       
 
\end{proof}

Next we consider convergence of relations. Let $A_n\subset X\times X$ be relations, $n\in\N$. We define the \emph{limit} of $A_n$ as the relation 
\[  A:=\{  (x,y)\in X:  \exists (x_n,y_n)\in A_n\mbox{ such that } \lim_{n\to\infty}(x_n,y_n)=(x,y)\mbox{ in }X\times X\},\]
and write $A=\lim_{n\to\infty} A_n$. Then  $A$ is a closed relation.  This a consequence of the following lemma applied to $Z=X\times X$. 
\begin{lem}\label{lem:3.7new}
Let $Z$ be a Banach space and $Z_n\subset Z$, $n\in\N$, subspaces.  	Then 
\[  Z_\infty := \{  x\in Z:\mbox{ there exist }x_n\in Z_n\mbox{ such that }\lim_{n\to\infty} x_n=x\} \]
is a closed subspace of $Z$.
\end{lem}	
\begin{proof}
It is obvious that $Z_\infty$ is a subspace of $Z$. We show that $Z_\infty$ is closed. Let $z\in \overline{Z_\infty}$. We construct  inductively over $k\in\N$ sequences $(y_n^k)_{n\in\N}$ and numbers $n_k\in\N$ such that $y_n^k\in Z_n$ for all $n\in \N$, $n_k<n_{k+1}$ and 
\[  \|z-y_n^k\|_Z<\frac{1}{k}\mbox{ for all }n\geq n_k.\]  
Let $k=1$. Choose $x\in Z_\infty$ such that $\|z-x\|_Z<\frac{1}{2}$. There exists $y_n^1\in Z_n$ such that $\lim_{n\to\infty} y_n^1=x$. Choose $n_1\in\N$ such that $\|y_n^1-x\|_Z<\frac{1}{2}$ for all $n\geq n_1$. Then $\|y_n^1-z\|_Z <1$ for all $n\geq n_1$. 

Now let $k>1$ and assume that the sequences $(y_n^m)_{n\in\N}$ and $n_m$ are constructed for $m\leq k-1$. Choose $x\in Z_\infty$ such that $\|x-z\|_Z<\frac{1}{2k}$. There exist $y_n^k\in Z_n$ such that $x=\lim_{n\to\infty} y_n^k$. Choose $n_k>n_{k-1}$ such that $\|x-y_n^k\|_Z<\frac{1}{2k}$ for all $n\geq n_k$. Then $\|y_n^k-z\|_Z<\frac{1}{k}$ for all $n\geq n_k$. This proves the inductive statement. Now we define $z_n\in Z_n$ as follows. For $n=1,\cdots, n_2-1$ we let $z_n=y_n^1$ and for $k>1$, $z_n=y_n^k$ for $n_k\leq n<n_{k+1}$. Then $z_n\in Z_n$ and $\|z_n-z\|<\frac{1}{k}$ for all $n\geq n_k$. Thus $z=\lim_{n\to\infty} z_n$ and consequently  $z\in Z_\infty$.     
\end{proof}	

It was pointed out  in \cite[Section 1.9]{BHS20} that $\lim_{n\to\infty}A_n$, which  can be defined for any sequence of relations, gives the right limit for resolvent convergence.  Slighly different versions of the following theorem are \cite[Theorem 3.7]{taiwan} and \cite[Theorem 1.9.4]{BHS20}.     
\begin{thm}\label{th:3.5}
Let $A_n$ be relations, $n\in\N$, $A=\lim_{n\to\infty} A_n$. Let $\lambda\in \rho(A_n)$  for all $n\in \N$ such that $\sup_{n\in\N} \|R(\lambda, A_n)\|<\infty$. The following assertions are equivalent:
\begin{enumerate}
	\item[(i)] $\ran (\lambda -A)$ is dense;
	\item[(ii)] $(R(\lambda, A_n))_{n\in\N}$ is strongly convergent. 
\end{enumerate}  	
In that case $\lambda\in\rho(A)$ and 
\[   R(\lambda,A)=\lim_{n\to\infty} R(\lambda,A_n)\mbox{ strongly}. \]
\end{thm}
\begin{proof}
	Let $c>0$ such that $\|R(\lambda, A_n)\|\leq c$ for all $n\in\N$. This implies that 
	\begin{equation}\label{eq:3.5}
		\|x\|\leq c\|\lambda x- y\| \mbox{ for all }(x,y)\in A.
	\end{equation}
In fact, let $(x,y)=\lim_{n\to\infty} (x_n,y_n)\in A$ with $(x_n,y_n)\in A_n$. Then 
$x_n=R(\lambda, A_n)(\lambda x_n -y_n)$. This implies \eqref{eq:3.5}. Proposition~\ref{prop:2.1} implies that $\ker (\lambda -A)=\{ 0\}$ and $\ran (\lambda -A)$ is closed. Now assume $(i)$. Then $\ran (\lambda -A)=X$. Thus $\lambda \in\rho(A)$. Let $v\in X$. We show that 
\[  \lim_{n\to\infty} R(\lambda,A_n)v=R(\lambda, A)v\]
There exists $(x,y)\in A$ such that $x-\lambda y=v$. Let $(x_n,y_n)\in A_n$ such  that $\lim_{n\to\infty} (x_n,y_n)=(x,y)$. Then 
\[  x_n=R(\lambda , A_n)(\lambda x_n-y_n)=R(\lambda, A_n)((\lambda x_n-y_n)-(\lambda x-y))+ R(\lambda,A_n)(\lambda x-y).\]
The first term converges to $0$. Since $\lambda x-y=v$, it follows that 
\[   \lim_{n\in\infty}R(\lambda , A_n)v=\lim_{n\to\infty}x_n=x=R(\lambda, A)v.\] 
We have shown that $(i)$ implies $(ii)$ and the additional assertion.\\
$(ii)\Rightarrow (i)$ Let $v\in X$. There exists $x$ such that $x_n:=R(\lambda ,A_n)v \to x$ as $n\to\infty$. Then, for each $n\in\N$,  there exists  $y_n\in X$ such that $(x_n,y_n)\in A_n$ and $\lambda x_n-y_n=v$. Hence $\lim_{n\to\infty} y_n=\lambda x-v$ and $(x,\lambda x-v)\in A$.  Thus 
\[   (x,v)=(x,\lambda x -(\lambda x -v))\in (\lambda -A).\]
Thus $v\in\ran (\lambda -A)$. We have shown that $\ran (\lambda -A)=X$.   
\end{proof} 
\section{$m$-dissipative relations }\label{sec:4}
Let $X$ be a Banach space over $\K=\R$ or $\C$. A relation $A\subset X\times X$ is called \emph{dissipative} if $\|\lambda x\|\leq \|\lambda x-y\|$ for all $(x,y)\in A$ and all $\lambda >0$. 

As in the case of operators, we have a characterization in terms of the duality map. For $x\in X$, let 
\[  dN (x):=\{  x'\in X':\|x'\|\leq 1, \langle x',x\rangle =\|x\|\}. \]
Then the following holds. 
\begin{prop}\label{prop:4.1}
Let $A\subset X\times X$ be a relation. Then $A$ is dissipative  if and only if for all $(x,y)\in A$ there exists $x'\in dN(x)$ such that $\re \langle x',y\rangle \leq 0$. 
\end{prop}
  \begin{proof}
  	The proof for operators is also valid for relations, see  e.g. \cite[Lemma 3.4.2]{ABHN11}.
  \end{proof} 
\begin{cor}\label{cor:4.2}
Let $H$ be a Hilbert space over $\K=\R$ or $\C$, $A\subset H\times H$ a relation.  Then $A$ is dissipative if and only if $\re \langle x,y\rangle_H\leq 0$ for all $(x,y)\in H$. 
\end{cor}
A relation $A$ is called \emph{$m$-dissipative} if in addition to dissipativity, the \emph{range condition} 
$\ran (\lambda -A)=X$ is satisfied for  some $\lambda >0$. As in the case  where $A$ is an operator, $m$-dissipativity can be characterized as follows.
\begin{prop}\label{prop:4.3}
	Let $A\subset X\times X$ be a relation. The following assertions are equivalent:
	\begin{itemize}
		\item[(i)] $A$ is $m$-dissipative;
		\item[(ii)] $(0,\infty)\in \rho(A)$ and $\|\lambda R(\lambda , A)\|\leq 1$ for all $\lambda>0$.   
	\end{itemize} 
\end{prop}  
\begin{proof}
	The proof of \cite[Theorem 3.4.5]{ABHN11} can easily be adapted. 
\end{proof}  
The following observation, easy to check,  is useful. 
\begin{lem}\label{lem:4.4}
Let $A$ be dissipative. Then also the closure $\overline{A}$ is dissipative. 
\end{lem}
In some cases the range condition is easier to prove for $\lambda=0$. And indeed, surjective and dissipative relations are $m$-dissipative. 
\begin{thm}\label{th:4.4}
Let $A\subset X\times X$ be a dissipative relation. If $\ran A=X$, then $A$ is $m$-dissipative and $0\in \rho(A)$.   
\end{thm}
Theorem~\ref{th:4.4} was proved by Sonja Thomaschewski in her thesis \cite[Theorem 3.4.13]{Tho03} if $A$ is a densely defined operator. It was used there to investigate non-autonomous equations. We will apply it to relations in Section~\ref{sec:6}. 

For the proof, we need some preparation.   
\begin{prop}\label{prop:4.5}
Let $A$ be an $m$-dissipative relation on $X$. If $0\neq u\in \ker A$, then there exists $u'\in \ker A'$ such that $\langle u',u\rangle\neq 0$.   
\end{prop}
\begin{proof}
	Let $0\neq u\in \ker A$. There exists $x'\in X'$ such that $\langle x',u\rangle =\|u\|$, $\|x'\|=1$. Using  Proposition~\ref{prop:4.3}, and since the dual unit ball  is compact for the weak* topology, there exists a directed set $I$ and a  convergent subset  $(y'_i)_{i\in I} $ of $(\lambda R(\lambda , A)'x')_{\lambda>0}$ in the following sense: there exists a mapping $\lambda:I\to(0,\infty)$ such that 
	\begin{itemize}
\item[(a)] $y'_i=\lambda_iR(\lambda_i,A)'x'$, $i\in I$;
\item[(b)] for all $\lambda_0>0$, there exists $i_0\in I$ such that  $\lambda_i\leq \lambda_0$ for all $i_0\prec i$ and 
\item[(c)]$\lim_{i\in I} y'_i=u'$ exists for the weak*-topology (see \cite[IV.2]{RS}). 
\end{itemize}
 Since $u\in \ker A$ one has $\lambda R(\lambda ,A)u=u$ for all $\lambda >0$. Thus 
\[  \langle y'_i,u\rangle =\langle \lambda_i R(\lambda_i, A)'x',u\rangle=\langle x',\lambda_i R(\lambda_i, A)u\rangle =\langle x',u\rangle=\|u\|  \]
for all $i\in I$. It follows that $\langle u',u\rangle =\|u\|$; in particular $u'\neq 0$. 
It follows from $(c)$ that $\lim_{i\in I}R(1,A)'y_i'=R(1,A)'u'$ for the weak*-topology. On the other hand, by the resolvent identity \eqref{eq:3.3}, 
\begin{eqnarray*}
R(1,A)'y_i'=  R(1,A)' \lambda_i R(\lambda_i,A)'x'& = & \frac{\lambda_i}{1-\lambda_i}R(\lambda_i,A)'x'-\frac{\lambda_i}{1-\lambda_i}R(1,A)'x'\\
 & = & \frac{y'_i}{1-\lambda_i}-\frac{\lambda_i}{1-\lambda_i}R(1,A)'x' .
\end{eqnarray*}
Using (b) and (c), $\lim_{i\in I} R(1,A)'y'_i=u'$ for the weak*-topology. We have shown that $R(1,A)'u'=u'$. This implies that $u'\in \ker A'$. 
 \end{proof}
\begin{proof}[Proof of Theorem~\ref{th:4.4}]
Let $A$ be a dissipative relation and assume that $\ran A=X$ . Then $\overline{A}$ is a closed dissipative relation and $\ran \overline{A}=X$.  It follows from Corollary~\ref{cor:2.7} that there exists $\lambda>0$ such that $\ran (\lambda-\overline{A})=X$. Thus $\overline{A}$ is $m$-dissipative. We show that $\ker \overline{A}=\{0\}$. Otherwise, by Proposition~\ref{prop:4.5}, there exists $0\neq u'\in \ker \overline{A}'$.  But 
 by Proposition~\ref{prop:2.2},
\[ \{0\} =X^\perp= (\ran \overline{A})^\perp =\ker \overline{A}'.\]
This is a contradiction since $u'\in \ker \overline{A}'$. Thus $\ker \overline{A}=\{ 0\}$. Now let $(x,y)\in \overline{A}$. By assumption there exists $x_1\in X$ such that $(x_1,y)\in A$. It follows that $x-x_1\in \ker \overline{A}$. Hence $x-x_1=0$. Thus $(x,y)=(x_1,y)\in A$. We have shown that $\overline{A}\subset A$.   
\end{proof}	
Next we consider maximal dissipative relations.  
\begin{defn}\label{del:4.6}
A relation $A\subset X\times X$ is \emph{maximal dissipative} if $A$ is dissipative and has no proper dissipative extension. 
\end{defn}
\begin{lem}\label{lem:4.7}
Each $m$-dissipative relation is maximal dissipative.  
\end{lem}
\begin{proof}
Let $A$ be $m$-dissipative and $A\subset B\subset X\times X$, $B$ dissipative. Let $(x,y)\in B$. Since $\ran (1-A)=X$, there exists $(x_1,y_1)\in A$ such that $x_1-y_1=x-y$. Thus $(x,x-y)\in 1-B$ and \[(x_1,x-y)=(x_1, x_1-y_1)\in 1-A\subset 1-B.\]
Hence $(x-x_1,0)\in 1-B$.  Since $B$ is dissipative, it follows that $x-x_1=0$. Hence $(x,y)=(x_1,y_1)\in A$.   
\end{proof}
We will now prove that the converse is true on a Hilbert space  $H$ over $\K=\R$ or $\C$. 
\begin{prop}\label{prop:4.8}
Let $A\subset H\times H$ be a dissipative relation. If $(\ran (1-A))^\perp\neq 0$, then there exists a proper dissipative extension $B$ of $A$. If $A$ is an operator, then one may choose as $B$ an operator.  	
\end{prop}
\begin{proof}
a) Let $R=\ran (1-A)$. Assume that $R^\perp\neq \{0\}$. Define $B:=\{  (x+v,y-v):(x,y)\in A, v\in R^\perp\}$. Then $A\subsetneq B$ and $B$ is dissipative. In fact,
\begin{eqnarray*}
 \re \langle x+v,y-v\rangle_H & =	 &\re \langle x,y\rangle_H-\re \langle x,v\rangle_H +\re \langle v,y\rangle_H - \langle v,v\rangle_H\\
  & \leq & -\re \langle x,v\rangle_H  + \re \overline{\langle y,v\rangle_H}\\
   & = & -\re \langle x,v\rangle_H  + \re {\langle y,v\rangle_H}\\
 & = & -\re \langle x-y,v\rangle_H\\
  & = & 0
\end{eqnarray*}
since $v\in (\ran (1-A))^\perp$. \\
b) We claim  that $\dom A\cap R^\perp =\{0\}$. Indeed, let $x\in \dom A\cap R^\perp$.  There exists $y\in X$ such that $(x,y)\in A$. Hence $\langle x-y,x\rangle_H=0$. Thus $\langle x,x\rangle_H =\re \langle y,x\rangle_H\leq 0$ and then  $x=0$. \\
c) Assume that $A$ is an operator. Then also the relation $B$ from $a)$ is an operator. In fact, let $(x,y)\in A$, $v\in R^\perp$ so that $(x+v,y-v)\in B$. If $x+v=0$, then by $b)$ $x=v=0$. Hence $(0,y)\in A$, and so $y=0$.  
\end{proof}
\begin{lem}\label{lem:4.9}
Let $A$ be a dissipative relation. Then $\overline{\ran (1-A)}=\ran (1-\overline{A})$. In particular, $A$ is closed if and only if  $\ran (1-A)$ is closed.    
\end{lem}
\begin{proof}
This follows from the inequality $\|x\|\leq \|x-y\|$ for all $(x,y)\in A$.  
\end{proof}
\begin{thm}\label{th:4.10}
A relation $A\subset H\times H$ is $m$-dissipative if and only if $A$ is maximal dissipative. 
\end{thm}
\begin{proof}
	By Lemma~\ref{lem:4.7},  $m$-dissipativity implies maximal dissipativity.  Conversely, let $A$ be maximal dissipative.  Then $\overline{A}$ is dissipative and  thus $A=\overline{A}$.  Then $\ran A$ is closed by Lemma~\ref{lem:4.9}. Now Proposition~\ref{prop:4.8}  implies that $\ran (1-A)=H$.    
\end{proof}
An application of Zorn's lemma shows that a dissipative relation has a maximal dissipative extension (which then is $m$-dissipative). Similarly, each dissipative operator has a maximal dissipative operator as extension. However, by an example of Phillips \cite[footnote 6 ]{Ph59}, there exists a maximal dissipative operator $A$ which is not closed . Thus, $\overline{A}$ is not an operator and $A$ is not $m$-dissipative. However, the following holds. 
\begin{thm}\label{th:4.11}
Let $A$ be a \emph{maximal dissipative operator} (i.e. if $A\subset B$ where $B$ is a dissipative operator, then $A=B$). Then $\overline{A}$ is an $m$-dissipative relation. 
\end{thm}     
\begin{proof}
Assume that $R:=\ran (1-\overline{A})	\neq H$. Then by Proposition~\ref{prop:4.8}, $A$ has a proper dissipative extension which is an operator. Thus $\ran (1-\overline{A})=H$, i.e. $\overline{A}$ is $m$-dissipative.    
\end{proof}	
For densely defined operators, things are different. The following holds. 
\begin{thm}\label{th:4.12}
Let $A\subset H\to H$ be an operator. The following assertions are equivalent:
\begin{itemize}
	\item[(i)] $A$ is $m$-dissipative;
	\item[(ii)] \begin{itemize}
		\item[(a)] $\dom A$ is dense and 
		\item[(b)] $A$ is a  maximal dissipative operator.   
	\end{itemize} 
\end{itemize}
\end{thm}
 \begin{proof}
 $(i)\Rightarrow (ii)$ $(a)$ follows from \cite[Proposition 3.3.8]{ABHN11} and $(b)$ from Lemma~\ref{lem:4.7}\\
 $(ii)\Rightarrow (i)$ Let $A$ be a maximal dissipative operator with dense domain. By Theorem~\ref{th:4.11}, $\overline{A}$ is $m$-dissipative. It is well-known that $\overline{A}$ is an operator (since $\dom A$ is dense), see \cite[Lemma 3.4.4]{ABHN11}. Thus $A=\overline{A}$ by maximality.   
 \end{proof}
It was Phillips who has shown that $(ii)\Rightarrow (i)$ in Theorem~\ref{th:4.12}. He used the Cayley transform. The more direct proof we give here is from \cite{RDV3}, where the operator case of Theorem~\ref{th:4.10} is proved.  
\section{Generation theorems}\label{sec:5}
In this section we establish generation theorems for $m$-dissipative relations. We want to use the following definition. Let $X$ be a Banach space over $\K=\R$ or $\C$. 
\begin{defn}\label{def:5.1}
A mapping $T:(0,\infty)\to {\mathcal L}(X)$ is a \emph{semigroup} if $T(t+s)=T(t)T(s)$ for all $t,s>0$. We speak of a  \emph{strongly continuous semigroup} if the mapping $T$ is strongly continuous. We call $T$ \emph{non-degenerate} if for $x\in X$,
\[ T(t)x=0\mbox{ for all }t\geq 0\mbox{ implies }x=0.\]  
\end{defn}    
A $C_0$-semigroup is a semigroup $T$ such that $\lim_{n\to\infty} T(t)=\Id$ strongly. This implies that $T$ is strongly continuous. Our definition deviates from the literature where the term "strongly continuous semigroup" is frequently used synonymously for $C_0$-semigroup.  

First we define the generator of a bounded strongly continuous semigroup. 
\begin{thm}\label{th:5.2}
Let $T:(0,\infty)\to {\mathcal L}(X)$ be a strongly continuous semigroup such that $\|T(t)\|\leq M$ for all $t>0$. Then there exists  a unique closed relation $A\subset X\times X$ such that $(0,\infty)\subset \rho(A)$ and 
\[ R(\lambda, A)x=\int_0^\infty  e^{-\lambda t} T(t)x \,dt\mbox{ for all }\lambda >0\mbox{ and }x\in X.\]
 We call $A$ the \emph{generator} of $T$. Thus $A$ is a closed relation and 
 \[\|\lambda R(\lambda,A)\|\leq M \mbox{ for all }\lambda >0.\]  
\end{thm}   
\begin{proof}
Let $R(\lambda)x:=\int_0^\infty e^{-\lambda t}T(t)x\,dt$.  Then $R:(0,\infty)\to{\mathcal L}(X)$ is a pseudo-resolvent by the proof of \cite[Theorem 3.1.7]{AB93}. By Proposition~\ref{prop:3.4}  there exists a unique closed relation $A\subset X\times X$ such that $(0,\infty)\subset \rho(A)$ and $R(\lambda)=R(\lambda, A)$ for all $\lambda >0$. Thus for all $x\in X$, $\lambda >0$, 
\begin{eqnarray*}
\|\lambda R(\lambda ,A)x\| & = & \|\int_0^\infty \lambda e^{-\lambda t} T(t)x dt\|\\
 & \leq & M\int_0^\infty \lambda e^{-\lambda t} dt \|x\|=M\|x\|. 	   
 \end{eqnarray*}
\end{proof}
In the situation of Theorem~\ref{th:5.2} one has
\begin{equation}\label{eq:5.1}
	\mul A=\{x\in X:T(t)x=0\mbox{ for all }t\geq 0\}.
\end{equation}
In fact, by Lemma~\ref{lem:3.3}, $\mul A=\ker R(\lambda, A)$ for all $\lambda>0$.  Thus the claim follows from the Uniquenss Theorem \cite[Theorem 1.7.3]{ABHN11}. In particular, the relation $A$ is an operator if and only $T$ is non-degenerate. 

If $T$ is a strongly continuous semigroup such that $\|T(t)\|\leq 1$ for all $t>0$, then  the generator $A$ of $T$ is an m-dissipative relation. We will see that the converse is true on Banach spaces with Radon--Nikodym property, but not in general.  We start investigating when a  strongly continuous $T$ has a limit as $t\to 0+$. This has the following consequences. 
\begin{prop}\label{prop:5.3}
	Let $T:(0,\infty)\to {\mathcal L}(X)$ be a strongly continuous semigroup. Assume that the strong limit 
	\[ T(0)=\lim_{t\to 0+} T(t)\]
	exists. Then $T(0)\in {\mathcal L}(X)$ is a projection such that $T(t)T(0)=T(0)T(t)$ for all $t\geq 0$. Let $X_1=T(0)X$, $X_0=\ker T(0)$. Then $X_0$ and $X_1$ are invariant under $T$, $X=X_0\oplus X_1$, 
	\[ t\mapsto T(t)_{|X_1}:[0,\infty)\to {\mathcal L}(X_1)\]
	is a $C_0$-semigroup and $T(t)_{|X_0}=0$ for all $t\geq 0$.     
\end{prop}      
\begin{proof}
	Let $P x=\lim_{t\to 0+} T(t)x$. Then $P\in {\mathcal L}(X)$ by the Banach--Steinhaus Theorem and $P^2=P$. Moreover, $T(t)P=PT(t)$. Thus $X_1=PX$ is invariant and for $x\in X_1$, $\lim_{t\to 0+} T(t)x=Px=x$. Thus $T_{|X_1}$ is a $C_0$-semigroup.  With the help of the Uniform Boundedness Principle, one finds $\delta>0$ and $M\geq 0$ such that $\|T(t)\|\leq M$ if $t\in [0,\delta]$. Let $x\in \ker P$, $t\in (0,\delta]$. Since for $0<s<t$, $T(t)x=T(t-s)T(s)x$, one has \[\|T(t)x\|\leq \limsup_{s\to 0+}M\|T(s) x\|=M\|Px\|=0\mbox{ for all }t>0\mbox{ and }x\in \ker P.\]     
\end{proof} 
Next we want to find criteria for $T(t)$ having a strong limit as $t\to 0+$. 
\begin{prop}\label{prop:5.4}
	Let $A\subset X\times X$ be a closed relation such that $(0,\infty)\subset \rho(A)$ and $\sup_{\lambda >0}\|\lambda R(\lambda ,A)\|<\infty$. 
\begin{enumerate}
\item[(a)] The following assertions are equivalent:
\begin{enumerate}
	\item[(i)] $P=\lim_{\lambda \to\infty}\lambda R(\lambda,A)$ exists for the strong operator topology;
	\item[(ii)] $\mul A+\dom A$ is dense in $X$;
	\item[(iii)] $\mul A\oplus \overline{\dom A }=X$.
\end{enumerate} 
In that case $P$ is the projection onto $\overline{\dom A}$ along the decomposition of (iii). 
\item[(b)]	If $X$ is reflexive, then the assertions $(i)-(iii)$ are automatically true.  
\end{enumerate}	 
\end{prop} 
\begin{proof}
	(a) 1) For $x\in \overline{\dom A}$ one has $\lim_{\lambda \to\infty} \lambda R(\lambda ,A)x=x$. In fact, by equicontinuity, it sufficies to prove the claim when $x\in \dom(A)$. Then there exists $y\in X$ such that $x=R(1,A)y$. Then 
	\[\lambda R(\lambda,A)x=\frac{\lambda}{\lambda-1}R(1,A)y-\frac{1}{\lambda-1}\lambda R(\lambda,A)y\to R(1,A)y=x\mbox{ as }\lambda\to\infty.\]  
	2) It follows from 1) that $\mul A\cap \overline{\dom(A)}=\{0\}$. \\
	$(i)\Rightarrow (iii)$ Let $x\in X$, then $Px=\lim_{\lambda\to\infty} \lambda R(\lambda ,A)x\in \overline{\dom A}$. Moreover, for all $\mu>0$, 
	\[R(\mu,A)(x-Px)=R(\mu,A)x-\lim_{\lambda\to\infty} \lambda R(\lambda,A)R(\mu,A)x=0\]
	by 1). Thus $x-Px\in \mul A$. We have shown that that $X=\mul A\oplus \overline{\dom A}$ and that $P$ is the projection onto $\overline{\dom A}$ along this decomposition.\\
	$(iii)\Rightarrow (ii)$ This is trivial.\\
	$(ii)\Rightarrow (i)$ This follows from 1).\\
	(b) Assume that $X$ is reflexive. We show that property (iii) is satisfied. Let $x\in X$. Then there exists $y\in X$ and a sequence $(\lambda_k)_{k\geq 0}$ such that $\lambda_k>0$ for all $k$ and $\lim_{k\to\infty}\lambda _k=\infty$ such that 
	\[  \lambda_k R(\lambda_k,A)x\rightharpoonup y\mbox{ as }k\to\infty\mbox{ (weak convergence).}\]  
	Since $R(1,\lambda)$ is weakly continuous, it follows that 
	\[ z_k:=R(1,A)(\lambda_k R(\lambda_k,A))\rightharpoonup R(1,A)y\mbox{ as }k\to\infty.\]
	But \[z_k=\frac{\lambda_k}{\lambda_k-1}R(1,A)x-\frac{1}{\lambda_k -1}\lambda_k R(\lambda_k,A)x\to R(1,A)x \mbox{ as } k\to\infty.\]
	Hence $x-y\in \ker R(1,A)=\mul A$. Since $y\in \overline{\dom A}$, we deduce that $x=(x-y)+y\in\mul A \oplus \overline{\dom A}.$   
\end{proof}
\begin{cor}\label{cor:5.5}
	Let $T:(0,\infty)\to{\mathcal L}(X)$ be a strongly continuous bounded semigroup with generator $A\subset X\times X$. The following assertions are equivalent:
	\begin{enumerate}
		\item[(i)] $T(0):=\lim_{t\to 0+} T(t)$ exists for the strong topology;
		\item[(ii)] $X=\mul A\oplus \overline{\dom A}$. 
	\end{enumerate} 
In that case $T(0)$ is the projection onto $\overline{\dom A}$ along the decomposition $(ii)$.   
\end{cor}
\begin{proof}
	$(i)\Rightarrow (ii)$ By a usual Abelian argument \cite[Theorem 4.1.2]{ABHN11},
	\[T(0) x=   \lim_{\lambda\to\infty} \lambda \int_0^\infty e^{-\lambda t}T(t) x\,dt=\lim_{\lambda \to\infty}\lambda R(\lambda,A)x\mbox{ for all }x\in X.\]
	Now $(ii)$ follows from Proposition~\ref{prop:5.4}. \\
	$(ii)\Rightarrow (i)$ Let $x\in \dom A$. Then there exists $y\in X$ such that $x=R(1,A)y$. Hence 
	\begin{eqnarray*}
		T(t)x & = & T(t)\int_0^\infty e^{-s}T(s) y \, ds\\
		 & = & e^{t}\int_t^\infty e^{-r}T(r) y \, dr.\\
		 & \to & \int_0^\infty e^{-r}T(r) y \,dr =x\mbox{ as }t\to 0+.
	\end{eqnarray*}  
Since $T(t)_{|\mul A}=0$ and $\mul A +\dom A$ is dense in $X$, $(i)$ follows.
\end{proof}
\begin{rem}
	The analogous assertion of Corollary~\ref{cor:5.5}  does not hold for $t\to\infty$. In fact, if $A$ generates a bounded $C_0$-semigroup $T$, then 
	$\overline{\ran A}\oplus \ker A=X$ if and only if $\lim_{t\to\infty}\frac{1}{t}\int_0^t T(r)x\, dr$ converges as $t\to\infty$ for all $x\in X$; but this does not imply strong converegnce of $T(t)$ as $t\to\infty$, even if $X$ is reflexive (consider the shift semigroup on $L^2(\R)$).   
\end{rem}
\begin{cor}\label{cor:5.6}
Assume that $X$ is reflexive. Let $A\subset X\times X$ be a closed relation. The following assertions are equivalent:
\begin{enumerate}
\item[(i)] $A$ generates a contractive, strongly continuous semigroup $T:(0,\infty)\to {\mathcal L}(X)$;
\item[(ii)] $A$ is $m$-dissipative. 	
\end{enumerate}	  
In that case $T(t)$ converges strongly as $t\to 0+$.  
\end{cor}
\begin{proof}
$(i)\Rightarrow (ii)$ We have noticed this after Theorem~\ref{th:5.2}; reflexivity is not needed for this implication. \\
$(ii)\Rightarrow (i)$ By Proposition~\ref{prop:5.4}, $X=\mul A\oplus \overline{\dom A}$. Consider $X_1=\overline{\dom A}$ and 
$A_1=A\cap (X\times X_1)$. Then $A_1$ is an $m$-dissipative operator which is densely defined. In fact, let $z\in \overline{\dom A}$. Then $\lim_{\lambda \to\infty} \lambda R(\lambda,A)z=z$ by part (a)1) of the proof of Proposition~\ref{prop:5.4}   and $R(\lambda ,A)z\in \dom A_1$ for all $\lambda >0$.   By the Hille--Yosida Theorem, $A_1$ generates a contractive $C_0$-semigroup $T_1$ on $X_1$. Then $T(t)(x_0+x_1):=T_1(t)x_1$ defines a strongly continuous semigroup $T:(0,\infty)\to{\mathcal L}(X)$ whose generator is $A$.  The strong convergence of $T(t)$ as $t\to 0+$ follows from Corollary~\ref{cor:5.5}. 
\end{proof}
From the proof of Proposition~\ref{prop:5.4} we also deduce the following corollary. 
\begin{cor}\label{cor:5.8new}
	  Let $A$ be a dissipative relation on a  Banach space $X$. Let $X_1=\overline{\dom A}$ and let $A_1=A\cap (X\times X_1)$ be the part of $A$ in $X_1$. Then $A_1$ is a densely defined $m$-dissipative operator on $X_1$, i.e. $A_1$ generates a contractive $C_0$-semigroup $T_1$ on $X_1$. In particular, each densely defined $m$-dissipative relation is an operator.    
\end{cor}
\begin{proof}
From Part (a) 1) in the proof of Proposition~\ref{prop:5.4} we know that $\lim_{\lambda\to\infty}\lambda R(\lambda,A)x=x$ for all $x\in X_1$.  Observe that for $x\in X_1$, $\lambda>0$, 
\[ (R(\lambda, A)x,\lambda R(\lambda, A)x-x)\in A\cap (X\times X_1)=A_1.\] 
Thus $\lambda R(\lambda, A)x\in \dom A_1$. This shows that $\dom A_1$ is dense in $X_1$. 
Moreover, $A_1$ is an operator. Since $R(\lambda,A)X_1\subset \dom A\subset X_1$ for all $\lambda >0$, it follows that $(0,\infty)\subset \rho(A_1)$. 
\end{proof}
Corollary~\ref{cor:5.6} is not true on arbitrary Banach spaces. An $m$-dissipative relation does not generate a strongly continuous semigroup, in general. However, it always generates an $1$-Lipschitz continuous integrated semigroup. This object is motivated by the following consideration.

Assume that $A\subset X\times X$ generates a contractive, strongly continuous   semigroup $T$ and let $S(t)x=\int_0^t T(s)x ds$ ($x\in X,$ $t\geq 0$).  Integration by parts shows that 
\[  R(\lambda, A)=\lambda \int_0^\infty e^{-\lambda t} S(t) dt \quad (\lambda >0)\]
where   $S(t)x=\int_0^t T(s)x ds$ for all $x\in X$, $t>0$. It turns out that this representation of the resolvent of an $m$-dissipative relation remains true on arbitrary Banach spaces. This is the content of the following two generation theorems. Let $X$ be an arbitary Banach space. 
\begin{thm}\label{th:5.5}
	Let $A\subset X\times X$ be an $m$-dissipative relation. Then there exists a unique function $S:[0,\infty)\to{\mathcal L}(X)$ such that 
	\begin{itemize}
		\item[(a)] $S(0)=0$
		\item[(b)] $\| S(t)-S(s)\|\leq |t-s|$ $(s,t\geq 0)$
		\item[(c)] $ R(\lambda, A)=\lambda \int_0^\infty e^{-\lambda t} S(t) dt \quad (\lambda >0)$. Moreover, the following functional equation holds:
		\item[(d)] $S(t)S(s)=\int_t^{t+s} S(r)dr - \int_0^{s} S(r)dr$ for all $t,s\geq 0$.  
	\end{itemize} 
We call $S$ the \emph{integrated semigroup} generated by $A$. 
\end{thm}    
\begin{proof}
It follows from \eqref{eq:3.2} that 
\[   \frac{R(\lambda ,A)^{(k)}}{k!}=(-1)^k R(\lambda, A)^{k+1}\]
	for $\lambda>0$, $k\in \N_0$.  Thus 
	\[  \frac{ \| \lambda^{k+1} R(\lambda, A)^{(k)}\| }{k!}\leq 1\mbox{ for all } \lambda >0, \, \, k\in\N_0 . \]
	The vector-valued version of Widder's theorem (\cite[Theorem 2.4.1 and Theorem 1.2.2]{ABHN11} or 
 	\cite[Theorem 1.1]{Ar87}) yields a function $S:[0,\infty) \to {\mathcal L}(X)$ such that $(a)$, $(b)$ and $(c)$ hold. It follows from the proof of \cite[Proposition 3.2.4]{ABHN11} (see also \cite[Theorem 3.1]{Ar87}) that $S$ satisfies $(d)$.  
 \end{proof}
 Note that $A$ is an operator if and only if $R(\lambda ,A)$ is injective for some (equivalently all $\lambda >0$), and this is equivalent to 
\begin{equation}\label{eq:5.2}
S(t)x=0\mbox{ for all }t\geq 0\Rightarrow x=0
\end{equation}   
by the Uniqueness Theorem for Laplace  transform or again the real representation theorem \cite[Theorem 2.2.1]{ABHN11}.  \\

Assume that $X$ has the Radon--Nikodym Property; i.e. each Lipschitz-continuous function $F:[0,\tau]\to X$ is differentiable almost everywhere for some (equivalently all) $\tau>0$. Then, by an argument in \cite[Lemma 6.3]{Ar87}, the functional equation in Theorem~\ref{th:5.5} (d)  implies that $S(\cdot)x\in C^1((0,\infty);X)$ for all $x\in X$. This leads to the following result. 
\begin{thm}\label{th:5.8n}
Assume that $X$ has the Radon--Nikodym Property. Let $A\subset X\times X$ be an $m$-dissipative relation. Then $A$ generates a strongly continuous semigroup $T:(0,\infty)\to{\mathcal L}(X)$ satisfying $\|T(t)\|\leq 1$ for all $t>0$.   
\end{thm} 
\begin{proof}
	The proof of \cite[Theorem 6.2]{Ar87} works also for relations.  
\end{proof}
Reflexive and separable dual Banach spaces have the Radon--Nikodym Property (see \cite[Section 1.2]{ABHN11}). In the situation of Theorem~\ref{th:5.8n}, the strong limit of $T(t)$ for $t\to 0+$ may not exist if $X$ is not reflexive. 
In fact, in  \cite[Example 6.4]{Ar87} a strongly continuous non-degenerate semigroup $T$ on a  separable dual Banach space  is constructed, which is not a $C_0$-semigroup.   Then $|||x|||:=\sup_{t>0} \|T(t)x\|$ defines an equivalent norm on $X$ for which $T$ is contractive. Thus the generator of $T$ is an $m$-dissipative  operator on $(X,|||\cdot|||)$.  \\

Next we prove the converse of Theorem~\ref{th:5.5}.
\begin{thm}\label{th:5.6}
Let $S:[0,\infty)\to {\mathcal L}(X)$ be a function such that $(a)$, $(b)$ and $(d)$ hold. Then there exists a unique $m$-dissipative relation $A$ such that $(c)$ holds.     
\end{thm} 
\begin{proof}
Let $R(\lambda):=\int_0^\infty e^{-\lambda t} S(t)dt $ for $\lambda >0$. Then $R(\lambda) \in {\mathcal L}(X)$ and $\| \lambda R(\lambda)\|\leq 1$ for all $\lambda>0$. In fact, let $x\in X$, $x'\in X'$ such that 
$\|x\|\leq 1, \|x'\|\leq 1$. Then there exists $f\in L^\infty (0,\infty)$ such that $\|f\|_\infty\leq 1$ and 
\[  \langle S(t)x,x'\rangle =\int_0^t f(s) ds \mbox{ for all }t\geq 0 \]
(see \cite[Proposition 1.2.3]{ABHN11}). Thus 
\[  |\langle x',\lambda R(\lambda) x\rangle | =\left|    \int_0^\infty \lambda e^{-\lambda t} f(t)dt \right| dt  \leq \int_0^\infty \lambda e^{-\lambda t} dt=1
\mbox{ for all }\lambda >0.\]    
By the proof of \cite[Proposition 3.2.4]{ABHN11}, $R:(0,\infty)\to{\mathcal L}(X)$ is a pseudo-resolvent. Now, by Proposition~\ref{prop:3.4} there exists a unique relation $A$ such that $(0,\infty)\subset \rho(A)$ and $R(\lambda)=(\lambda-A)^{-1}$ for all $\lambda >0$.    
\end{proof}	
Our next aim is to describe an evolutionary problem  governed by the integrated  semigroup $S$. 

We need the following version of the uniqueness theorem.  
\begin{thm}[Uniqueness theorem]\label{th:5.7}
Let $Z$ be a Banach space  and $Y\subset Z$ a closed subspace. Let $f:(0,\infty)\to Z$ be continuous such that $\|  \int_0^t f(s)ds \|\leq Me^{wt}$, where $w\in\R$, $M\geq 0$. Assume that 
\[  \widehat{f}(\lambda) :=\int_0^\infty  e^{-\lambda t} f(t)dt\in Y\mbox{ for all }\lambda >w.\]
Then $f(t)\in Y$ for all $t> 0$. 
\end{thm}
   \begin{proof}
   	Denote by $q:Z\to Z/Y$ the quotient map. Then $\widehat{q\circ f}(\lambda)=q\circ \widehat{f} (\lambda)=0$ for $\lambda >w$. Now the uniqueness theorem \cite[Theorem 1.7.3]{ABHN11} implies that $q(f(t))=0$ for all $t> 0$. 
   \end{proof}
Next we establish a fundamental formula for the integrated semigroup generated by an $m$-dissipative relation. The proof is very different from the proof in the operator case, cf. \cite[Lemma 3.2.2]{ABHN11}. 
\begin{prop}\label{prop:5.8}
Let $A$ be an $m$-dissipative relation and $S$ the integrated semigroup generated by $A$. Then 
\begin{equation}\label{eq:5.3}
	\left(   \int_0^t S(s) xds, S(t)x-tx \right)\in A \mbox{ for all }x\in X,\,\, t\geq 0.
\end{equation}	
\end{prop} 
\begin{proof}
	For $\lambda >0$, $x\in X$, 
	\[  (R(\lambda ,A)x, \lambda R(\lambda, A)x-x)\in A.\]
	Note that 
	\[R(\lambda ,A)x=\lambda \int_0^\infty e^{-\lambda t} S(t) dt=\lambda^2 \int_0^\infty e^{-\lambda t}\int_0^t S(s) x\,ds\, dt  \]
	and $\lambda^2 \int_0^\infty e^{-\lambda t}t dt=1$.  Thus 
	\[  \lambda^2 \int_0^\infty  e^{-\lambda t}\left(   \int_0^t S(s)xds , S(t)x -tx \right)dt    \in A ,  \]
	where the integrand takes values in $X\times X$. Now the claim follows from  Theorem~\ref{th:5.7}.  
\end{proof}
Now we can describe $m$-dissipative relations by a well-posedness result. Let $A$ be a closed relation. Given $x\in X$ we consider the problem 
\begin{equation}\label{eq:5.4}
\dot{u} (t)\in A u(t) +x,\, \, t>0,\,\, u(0)=0.
\end{equation}	
A \emph{mild solution} of \eqref{eq:5.4} is a function $u\in  C([0,\infty),X)$ such that $u(0)=0$ and 
\[  \left(     \int_0^t u(s)ds, \,\, u(t)-tx\right) \in A\mbox{ for all }t\geq 0. \] 
If $u$ is a \emph{classical solution} of \eqref{eq:5.4}, i.e. $u\in C^1 ((0,\infty),X )\cap C([0,\infty),X)$ and $u$ satisfies   \eqref{eq:5.4}, then $u$ is a mild solution.  Conversely, a mild solution which is in $C^1((0,\infty), X)$ is a classical solution. 
 \begin{thm}\label{th:5.9}
 Assume that $A$ is an $m$-dissipative relation. Then for all $x\in X$ Problem \eqref{eq:5.4} has a unique mild solution $u$. In fact $u(t)=S(t)x$  for all  $t\geq 0$, where $S$ is the integrated semigroup generated by $A$. In particular, 
 \begin{equation}\label{eq:5.5}
 \| u(t)-u(s)\|\leq |t-s| \|x\|\mbox{ for all }s,t\geq 0. 	
 \end{equation}	  
 \end{thm}
\begin{proof}
	Given $x\in X$, by \eqref{eq:5.3}, the function $u(\cdot)=S(\cdot)x $ is a mild solution of \eqref{eq:5.4}. In order to prove uniqueness, let $u$ be the difference of two mild solutions of \eqref{eq:5.4}. Let $w(t)=\int_0^t u(s)ds$. Then $(w(t),\dot{w}(t))\in A$ for all $t\geq 0$.  Since $w(t)\in \dom A$, it follows that  $\dot{w} (t)\in \overline{\dom A}=:X_1$. Thus $(w(t),\dot{w}(t)) \in A_1=A\cap (X\times X_1)$. Since by Corollary~\ref{cor:5.8new} the operator $A_1$ generates a $C_0$-semigroup $T$ on $X_1$, and since $w(0)=0$, it follows that $w(t)=0$  for all $t\geq 0$. This implies that $u(t)=0$ for all $t\geq 0$.   
\end{proof}
Also the converse of Theorem~\ref{th:5.9} holds, i.e. $m$-dissipative relations can be characterized by the well-posedness of Problem \eqref{eq:5.4}.  
\begin{thm}\label{th:5.10}
Let $A$ be a closed relation. Assume that for all $x\in X$ Problem~\eqref{eq:5.4} has a unique mild solution $u$.  Assume also that this solution satisfies \eqref{eq:5.5}.  Then $A$ is $m$-dissipative. 
\end{thm}
\begin{proof}
For $x\in X$, let $S(t)x:=u(t)$ where $u$ is the mild solution of \eqref{eq:5.4}. Then, applying the closed graph theorem in the Fr\'echet space  $C([0,\infty);X)$, one sees that $S(t)\in {\mathcal L}(X)$, and moreover $S(0)=0$. Since $u$ satisfies \eqref{eq:5.5}, it follows that 
\begin{equation}\label{eq:5.6}
\|S(t)-S(s)\|\leq |t-s|\mbox{ for all }t,s\geq 0.
\end{equation}   
We show that 
\begin{equation}\label{eq:5.7}
S(t)S(s)=\int_s^{t+s} S(r)dr -\int_0^t S(r)dr.
\end{equation}
Let $s>0$, $x\in X$. Let 
\[  w(t) =\int_0^{t+s} S(r)xdr - \int_0^s S(r)xdr-\int_0^t S(r)xdr .\]
Note that 
\begin{equation}\label{eq:5.8}
\left(  \int_0^\tau S(r)x dr, S(\tau)x-\tau x  \right) \in A\mbox{ for all }\tau\geq 0.
\end{equation}
Using this for $\tau=t,s,t+s$, one deduces that 
\[  (w(t),\dot{w}(t)-S(s)x)=(w(t), S(t+s)x-S(t)x-S(s)x)\in A\mbox{ for all  }t\geq 0.\]
Thus $\dot{w}(t)\in Aw(t)+S(s)x$. 
This shows that $w(t)=S(t)S(s)x$ for all $t\geq 0$. Thus \eqref{eq:5.7} is proved.

Let $R(\lambda):=\lambda \widehat{S} (\lambda)$, $\lambda>0$, where 
\[   \widehat{S} (\lambda):=\int_0^\infty  e^{-\lambda t}S(t)dt.\]
 It follows from the proof of \cite[Proposition 3.2.4 ]{ABHN11} that $R:(0,\infty)\to{\mathcal L}(X)$ is a pseudo-resolvent. Thus there exists a closed relation $B\subset X\times X$ such that $(0,\infty)\subset \rho(B)$ and 
 $R(\lambda)=R(\lambda,B)$ for all $\lambda>0$. It follows from \eqref{eq:5.6} that $\|\lambda R(\lambda, B)\|=\|\lambda \widehat{S}(\lambda)\|\leq 1$ for all $\lambda >0$.  Thus $B$ is $m$-dissipative.
 
  We now show that $A=B$. Let $x\in X$. Taking Laplace transforms of \eqref{eq:5.8}, one obtains 
  \[   \left(     \frac{1}{\lambda} \widehat{S}(\lambda) x, \widehat{S} (\lambda)x-\frac{x}{\lambda^2}   \right) \in A\mbox{ for all }\lambda>0. \]
  Thus    \[   \left(     \frac{1}{\lambda} \widehat{S}(\lambda) x, \frac{x}{\lambda^2}   \right)\in (\lambda -A)\mbox{ for all }\lambda >0.\]   
  Since $(\lambda -A)\subset X\times X$ is a subspace, $(\lambda \widehat{S}(\lambda)x,x)\in (\lambda-A)$ for all $\lambda>0$, i.e. $(x,\lambda \widehat{S} (\lambda)x)\in (\lambda-A)^{-1}$ for all $\lambda >0$. We have shown that $(\lambda -B)^{-1}\subset (\lambda -A)^{-1}$ for all $\lambda >0$.  But $(\lambda -A)^{-1}$ is an operator. In fact, let $(0,x)\in (\lambda -A)^{-1}$. Then $(x,0)\in (\lambda -A)$ and  so there exists $y\in Ax$ such that  $\lambda x-y=0$. Let $u(t)=\frac{1}{\lambda} (e^{\lambda t} -1)x$. Then $u$ is a classical solution of \eqref{eq:5.4} and so a mild solution. Thus $u$ satisfies \eqref{eq:5.5} and so $\|\dot{u}\|\leq \|x\|$ for all $t\geq 0$; i.e. 
  $\|e^{\lambda t}x\|\leq \|x\|$ for all $t> 0$. This implies that $x=0$. We have shown that $(\lambda -A)^{-1}$ is an operator and so $(\lambda -B)^{-1}=(\lambda -A)^{-1}$ for all $\lambda >0$. This implies that $A=B$.     
\end{proof}
\begin{cor}\label{cor:new}
Let $A\subset X\times X$ be a closed relation. The following assertions are equivalent:
\begin{itemize}
	\item[(i)] $A$ generates a strongly continuous semigroup of contractions $T:(0,\infty)\to {\mathcal L}(X)$;
	\item[(ii)] for all $x\in X$, Problem \eqref{eq:5.4} has a unique classical solution $u$, and this solution satisfies \eqref{eq:5.5}.     
\end{itemize}  
\end{cor}
\begin{proof}
$(ii)\Rightarrow (i)$ We know from Theorem~\ref{th:5.10} that $A$ is $m$-dissipative and from Theorem~\ref{th:5.5} that $A$ generates an integrated semigroup  $S$ satisfying \eqref{eq:5.6}.  For $x\in X$, $u(t):=S(t)x$ is the unique mild solution of \eqref{eq:5.4}. Since $u$ is a classical solution, it follows that $S(\cdot)x\in C^1((0,\infty);X)$ for all $x\in X$. Let 
$T(t)x=\frac{d}{dt}  S(t)x$. Then $T:(0,\infty) \to {\mathcal L}(X)$ is strongly continuous.  Taking the derivative in Theorem~\ref{th:5.5} (d), one sees that $T$ is a semigroup. Integration by parts shows that \[R(\lambda,A)x=\int_0^\infty \lambda e^{-\lambda t} S(t) x dt = \int_0^\infty e^{-\lambda t} T(t) x dt\mbox{ for all }\lambda>0.\]
Thus $A$ is the generator of $T$.\\
$(i)\Rightarrow (ii)$ Let $S(t)x=\int_0^t T(s) x ds$. Then 
\[R(\lambda,A)=\int_0^\infty \lambda e^{-\lambda t} S(t)dt \mbox{ for all }\lambda>0.  \]
For $x\in X$, $u(t)=S(t)x$ defines the unique mild solution of \eqref{eq:5.4}. Since $u\in C^1((0,\infty);X)$, it is a classical solution. 
\end{proof}

Next we investigate convergence of $m$-dissipative relations. Recall that for a sequence $A_n\subset X\times X$, $n\in \N$, we define the relation
\[  \lim_{n\to\infty} A_n =\{   (x,y)\in X\times X:\exists (x_n,y_n)\in A_n,x_n\to x,y_n\to y\mbox{ as }n\to\infty\}. \]  
Recall that $\lim_{n\to\infty}A_n$ is closed by Lemma~\ref{lem:3.7new}. 
This allows a particularly simple formulation of the following result of Trotter--Kato type (see \cite[Theorem 3.6.1 and Proposition 3.6.2]{ABHN11} for related semigroup versions).
\begin{thm}\label{th:5.11}
Let $A_n,A$ be $m$-dissipative relations and $S_n,S$ the integrated semigroups generated by $A_n$ and $A$ respectively. The following assertions are equivalent:
\begin{itemize}
	\item[(i)] For each $x\in X$, $\lim_{n\to\infty} S_n(t)x=S(t)x$ uniformly on $[0,T]$ for all $T>0$;
	\item[(ii)] $\lim_{n\to\infty} R(\lambda, A_n)=R(\lambda,A)$ strongly for all $\lambda>0$;
	\item[(iii)] there exists $\lambda\in \rho(A)\cap \bigcap_{n\in\N}\rho(A_n)$ such that $R(\lambda, A_n)\to R(\lambda, A)$ strongly;
	\item[(iv)] if $\mu\in \K$ such that $\overline{\ran (\mu-A)}=X$, $\mu\in \rho(A_n)$ for all $n\in \N$ and $\sup_{n\in \N} \|R(\mu, A_n)\|<\infty$, then $\mu\in \rho(A)$ and $\lim_{n\to\infty}R(\mu,A_n)=R(\mu,A)$ strongly; 
	\item[(v)] $A=\lim_{n\to\infty} A_n$. 	   
\end{itemize}  
\end{thm}  
\begin{proof}
$(i)\Rightarrow (ii)$ Since \[R(\lambda, A_n) x=\int_0^\infty \lambda e^{-\lambda t}S_n (t)x dt \mbox{ and } R(\lambda, A) x=\int_0^\infty \lambda e^{-\lambda t} S(t)x dt,\] the dominated convergence theorem shows that 
\[R(\lambda , A_n) x\to R(\lambda, A)x.\]
$(ii)\Rightarrow (iii)$ is trivial.\\
$(iii)\Rightarrow (v)$ Let $B=\lim_{n\to\infty} A_n$. We deduce from Theorem~\ref{th:3.5} that $\lambda\in \rho(B)$ and $R(\lambda, B)=\lim_{n\to\infty} R(\lambda, A_n)$ strongly for $\lambda>0$. Thus $R(\lambda, B)=R(\lambda, A)$. Hence $A=B$. \\
$(v)\Rightarrow (iv)$ This is Theorem~\ref{th:3.5}. \\
$(iv)\Rightarrow (ii)$ If $\lambda>0$, then  $\lambda\in \rho(A)$ as well as $\lambda\in\rho(A_n)$ and $\|R(\lambda, A_n)\|\leq \frac{1}{\lambda}$ for all $n\in\N$. \\
$(ii)\Rightarrow (i)$ This follows from \cite[Theorem 1.1]{taiwan}. 
 \end{proof}
It should be emphasized that even if in Theorem~\ref{th:5.11} each $A_n$ generates a $C_0$-semigroup $T_n$, $(T_n(t))_{ n\in N}$ does not converge strongly, only the integrated semigroup $(S_n(t))_{n\in\N}$  does.  Here is a very simple example.
\begin{exam}\label{ex:5.12}
Let $X=\C$, $A_n=\{ (x,inx): x\in \R\}$. Then $A_n$ generates the $C_0$-semigroup $T_n$ on $\C$ given by 
$T_n(t)x=e^{int}x$. Moreover, one has $\lim_{n\to\infty} A_n =\{0\}\times \C=:A$, which is an 
 $m$-dissipative relation.  In fact, $R(\lambda, A)=0$ for all $\lambda >0$, and $S(t)=0$ for all  $t>0$.  Let $S_n (t)=\int_0^t T_n(s) ds =\frac{1}{in} (e^{int} -1)$.   Thus $\lim_{n\to\infty} S_n (t) =0$ as asserted by Theorem~\ref{th:5.11}.  Moreover, $(T_n(t))_n$ does not converge  unless $t\in 2\pi\Z$. This example was mentioned in \cite[Example 1.4]{taiwan} in terms of pseudo-resolvents. Here it is instructive to identify $\lim_{n\to\infty} A_n$. 
\end{exam}    
In Section~\ref{sec:7} we will see that the situation is much better in the holomorphic case. 

The result of this section can be easily generalized. If $T:(0,\infty)\to {\mathcal L}(X)$ is a strongly continuous  semigroup such that $\sup_{0<t\leq 1} \|T(t)\|<\infty$, then there exist $M\geq 0,\omega \in \R$ such that $\|T(t)\leq Me^{\omega t}$. Thus we may define the generator $A$ as in Definition~\ref{def:5.1} considering merely $\lambda >\omega$. Then $A-\omega$ generates  the bounded semigroup $T$. Moreover 
\[ |x|:=\sup_{t>0}  \|e^{-\omega t} T(t)x\|\]
defines an equivalent  norm making $A-\omega $ $m$-dissipative.  Conversely, given a closed relation $A$ which satisfies the Hille--Yosida condition, by the proof of \cite[Lemma 3.5.4]{ABHN11}, we obtain an equivalent norm making $A-\omega $ $m$-dissipative.   
\section{Domain convergence for the heat equation with Dirichlet boundary conditions}\label{sec:6}
Throughtout this section we choose $\K=\R$. Let $\Omega\subset \R^d$ be open and bounded. We will consider the Laplacian with Dirichlet boundary conditions on $\Omega$. Since we are interested in convergence results when $\Omega$ varies, we consider a large open ball $B$ such that $\overline{\Omega}\subset B$ and let $X=C(\overline{B}):=\{f:\overline{B}\to\R: f\mbox{ continuous}\}$ with the supremum norm 
\[  \|f\|_\infty:=\sup_{x\in \overline{B}}|f(x)|.\]  
We let \[C_0(\Omega):=\{ u\in C(\overline{B}): u(x)=0\mbox{ for all }x\in \overline{B}\setminus \Omega \}\]
Then $C_0(\Omega)$ is a closed subspace of $C(\overline{B})$. 

Now we define the \emph{Dirichlet-Laplacian} $A_\Omega$ with respect to $\Omega$  as a relation in $C(\overline{B})\times  C(\overline{B})$ by 
\[  A_\Omega := \{   (u,f)   : u\in C_0(\Omega), f\in C(\overline{B}), \Delta u=f\mbox{ in }{\mathcal D}(\Omega)'\}.\]
Here ${\mathcal D}(\Omega):=C^\infty_c(\Omega)$ is the space of all test functions, and to say that $\Delta u=f$ in  ${\mathcal D}(\Omega)'$ means that 
\[  \int_B u\Delta \varphi =\int_B f\varphi\mbox{ for all }\varphi\in {\mathcal D}(\Omega).\]
It is obvious that $A_\Omega\subset C(\overline{B})\times   C(\overline{B})$ is a closed relation. We assume furthermore that $\Omega$ is \emph{Dirichlet regular}; i.e. for all $g\in C(\partial \Omega)$ there exists $u\in C^2(\Omega)\cap C(\overline{\Omega})$ such that $\Delta u=0$, $u_{|\partial \Omega} =g$. This condition is very well understood. For example if $\Omega$ has Lipschitz boundary, then it is Dirichlet regular; for $d=2$  it suffices that $\Omega$ is simply connected. 
\begin{thm}\label{th:6.1}
	Assume that $\Omega$ is Dirichlet regular. Then $A_\Omega$ is $m$-dissipative and $0\in\rho(A_{\Omega})$.   
\end{thm}
To prove dissipativity we need the following maximum principle for the distributional Laplacian. For proving $m$-dissipativity we will use the Lumer--Phillips Theorem in the version of Theorem~\ref{th:4.4}.  
\begin{prop}\label{prop:6.2}
Let $\mathcal U$ be an open neighborhood of $x_0\in \R^d$. Let $u\in C({\mathcal U})$ such that $\Delta u\in C({\mathcal U})$. If $u(x_0)=\max_{x\in {\mathcal U}} u(x)$, then $\Delta u(x_0)\leq 0$.     
\end{prop}
\begin{proof}
	This is well-known if $u\in C^2(\Omega)$. To be complete we give a proof in the general case we need here. Let $(\rho_n)_{n\in\N}$ be a \emph{mollifying sequence}, i.e. $0\leq \rho_n\in {\mathcal D}(\R^d)$, 
supp$\rho_n\subset B(0,1/n)$, $\int_{\R^d}\rho_n =1$ for all $n\in \N$. Let $r>0$ such that $K=\overline{B(x_0,r)}\subset \Omega$. Define, for $n>1/r$, $u_n=\rho_n\star u$. Then  $u_n\to u$ uniformly on $K$. Consequently there exists a sequence $(x_n)_{n\in \N}\subset K$ such that $x_n\to x_0$ and $u_n(x_n)=\max_K u_n(x)$ for all $n\in\N$. By the classical result $\Delta u_n (x_n)\leq 0$.  Hence 
\[  (\Delta u)(x_0)=\lim_{n\to\infty} (\rho_n\star \Delta u)(x_n)=\lim_{n\to\infty} \Delta (\rho_n\star u)(x_n)\leq 0. \]
\end{proof}
\begin{proof}[Proof of Theorem~\ref{th:6.1}]
We first prove that $A_\Omega$ is dissipative. Let $(u,f)$ be in  $A_\Omega$. Since $u=0$ on $B\setminus  \overline{\Omega}$, there exists $x_0\in \Omega$ such that $|u(x_0)|=\|u\|_\infty$. \\
\emph{First case}: $u(x_0)>0$. Then $u(x_0)=\max_{x\in\R^d}u(x)$.  Then $\delta_{x_0}\in dN(u)$ (see Section~\ref{sec:4}). Then $\langle \delta_{x_0},f\rangle=(\Delta u)(x_0)\leq 0$ by Proposition~\ref{prop:6.2}.\\
\emph{Second case}: $u(x_0)<0$. Then $(-u,-f)\in A$ and $-u(x_0)=\|-u\|_\infty$.  From the first case we deduce that   $\langle \delta_{x_0},-f\rangle \leq 0$. Hence $\langle -\delta_{x_0}, f\rangle\leq 0$ and $-\delta_{x_0}\in dN(u)$. \\
\emph{Third case}: $u(x_0)=0$. Then $u=0$. Choose $x'_0=0$. Then $x'_0\in dN(u)$. 

In all the three cases we found $x_0'\in dN(u)$ such that 
$\langle x_0',f\rangle\leq 0$. We deduce from Proposition~\ref{prop:4.1} that $A_\Omega$ is dissipative. 

In order to prove $m$-dissipativity, by Theorem~\ref{th:4.4}, it suffices to show that $A_\Omega$ is surjective.   

Let $E_d\in L^1_{loc}(\R^d)$ be the Newtonian potential. Then for $f\in C_c(\R^d)$, $w=E_d\star f\in C^1(\R^d)$ and $\Delta w=f$ in the sense of distributions.  Now let $f\in C(B)$, extend $f$ to a function $\widetilde{f}\in C_c(\R^d)$. Then $w=E_d \star \widetilde{f}\in C^1(\R^d)$ and $\Delta w=\widetilde{f}$. Let $h\in C(\overline{\Omega})\cap C^2(\Omega)$ such that $\Delta w=0$ and $h_{|\partial \Omega} =w_{|\partial \Omega}$. Then $u=w-h\in C(\overline{\Omega})$, $u_{|\partial \Omega}=0$ and $\Delta v=f$ in ${\mathcal D}(\Omega)'$.  Extending $u$ by $0$ outside $\overline{\Omega}$ we obtain a  pair $(u,f)\in A_\Omega$. Thus $f\in \ran A_\Omega$. This completes the proof.       
\end{proof}
Next we want to study domain convergence. Let $\Omega$, $\Omega_n$ be open sets such that $\overline{\Omega}$,  $\overline{\Omega_n}\subset B$, where $B$ is a large open ball in $\R^d$. Following \cite[Definition 3.6]{AD08} we say that $(\Omega_n)_{n\in\N}$ \emph{converges} to $\Omega$ as $n\to\infty$ and write $\Omega_n\to\Omega$ if 
\begin{itemize}
	\item[(a)] for all compact $K\subset \Omega$ there exists $n_0\in \N$ such  that $K\subset \Omega_n$ for all $n\geq n_0$ and 
	\item[(b)] $\lambda (\Omega_n\setminus \overline{\Omega}) \to 0$ as $n\to\infty$. 
	\end{itemize}
Here $\lambda(\Omega_n\setminus \overline{\Omega})$ is the first eigenvalue of the Dirichlet Laplacian (see \cite[(2.6)]{AD08}).  
We remark that $(b)$ is satisfied whenever $|\Omega_n\setminus \overline{\Omega}|\to 0$ as $n\to\infty$.  Here $|F|$ stands for the Lebesgue measure of a Borel set $F\subset \R^d$. 

We say that an open bounded set is \emph{stable} if 
\[  H_0^1(\Omega)=H_0^1(\overline{\Omega}):=\{  u_{|\Omega}:u\in H^1(\R^d), u(x)=0\mbox{ for all }x\in \R^d\setminus \overline{\Omega} \}. \] 
If $\Omega$ has continuous boundary in the sense of graphs, then $\Omega$ is stable. Note that stability is independent of Dirichlet regularity. The Lebesgue cusp yields an example of a bounded open set with continous boundary which is not Dirichlet regular. Whereas each bounded open subset $\Omega$ of $\R$ is Dirichlet regular, the set $\Omega=(0,1)\cup (1,2)$ is not stable. 

Recall that for an open, bounded, Dirichlet regular set, the relation $A_\Omega\subset C(\overline{B})\times C(\overline{B})$ is $m$-dissipative. Thus $R(\lambda, A_\Omega)\in {\mathcal L}(C(\overline{B}))$ and 
$\|\lambda R(\lambda, A_\Omega)\|\leq 1$ for $\lambda>0$.  

Denote by $S_\Omega:[0,\infty)\to {\mathcal L}(C(\overline{B}))$ the integrated semigroup generated by $A_\Omega$.  

Now we can formulate the main result of this section. 
\begin{thm}\label{th:6.3}
Let $\Omega_n, n\in\N$, and $\Omega$ be bounded open sets, all Dirichlet regular. Suppose that $\overline{\Omega_n}, \,\overline{\Omega}\subset B$ for all $n\in\N$, and that $\Omega$ is stable.  Finally, suppose that $\Omega_n\to\Omega$ as $n\to\infty$. Then  
\begin{itemize}
	\item[(a)] $\lim_{n\to\infty} R(\lambda, A_{\Omega_n})=R(\lambda,A_\Omega)$ strongly for all $\lambda>0$;
	\item[(b)] $\lim_{n\to\infty} S_{\Omega_n}(t)f=S_\Omega(t)f$ in $C(\overline{B})$ uniformly on $[0,T]$ for all $T>0$, $f\in C(\overline{B})$; 
	\item[(c)] for all $(u,f)\in A_\Omega$ there exist $(u_n,f_n)\in A_\Omega$ for $n\in\N$ such that $u_n\to u$, $f_n\to f$ in $C(\overline{B})$.  
\end{itemize}
\end{thm}
\begin{proof}
	By \cite[Theorem 5.6]{AD08}, $R(0,\Delta_{\Omega_n})\to R(0,\Delta_\Omega)$ strongly. Now $(a),(b)$ and $(c)$ follow from Theorem~\ref{th:5.11}. 
\end{proof}
In terms of the solution of the inhomogeneous heat equation, Theorem~\ref{th:6.3} gives the following stability result. 
\begin{cor}\label{cor:6.4}
Under the assumptions of Theorem~\ref{th:6.3}, the following holds. Let $f\in C(\overline{B})$ be given and denote by $u,u_n\in C([0,\infty), C(\overline{B}))$ the mild solutions of 
\[  \dot{u}(t)\in A_\Omega u(t) +f, \,\, t\geq 0\mbox{ and }u(0)=0\] 
\[ \dot{u_n}(t)\in A_{\Omega_n} u_n(t) +f, \,\, t\geq 0 \mbox{ and }u_n(0)=0. \]
Then $u_n(t)\to u(t)$ 
as $n\to\infty$, uniformly on $[0,T]$ and for all $T>0$.  
\end{cor} 
\begin{proof} 
	One has $u_n(t)= S_{\Omega_n}(t)f $ and $u(t)=S_{\Omega}(t)f $ by Theorem~\ref{th:5.9}. So Corollary~\ref{cor:6.4} follows directly from Theorem~\ref{th:6.3}.  
\end{proof}
The corollary shows in particular that $u_n(t)_{|\overline{B}\setminus \Omega}\to 0$ as $n\to\infty$ uniformly for $t\in [0,T]$ and for all $T>0$. 

It is interesting that the Lumer--Phillips Theorem with surjectivity as range condition (Theorem~\ref{th:4.4}) allows the following perturbation of $A_\Omega$.  
\begin{prop}\label{prop:6.5}
Let $\Omega\in\R^d$ be open, bounded and Dirichlet regular. Let $B$ be a ball such that 
$\overline{\Omega} \subset B$. If $m\in C(\overline{B})$ such that $m(x)\neq 0$ for all $x\in \overline{B}$, then 
\[ mA_\Omega:=\{  (u,mf):(u,f)\in A_\Omega\}\]
is $m$-dissipative. 
\end{prop}  
\begin{proof}
	As for $A_\Omega$, one sees that $mA_\Omega$ is dissipative. Since $\ran A_\Omega=C(\overline{B})$ and $\frac{1}{m}\in C(\overline{B})$, we have also $\ran (mA_\Omega)=C(\overline{\Omega})$. Now the claim follows from Theorem~\ref{th:4.4}. 
\end{proof}

\section{Relations generating a holomorphic semigroup}\label{sec:7}
In contrast to $m$-disspativity, the usual holomorphic estimate of a closed relation yields a holomorphic semigroup without any restriction on the Banach space. Let $X$ be a complex Banach space. For $0<\alpha \leq \pi$ we consider the open sector
\[  \Sigma_\theta:=\{  re^{i\theta}:r>0,\theta\in (-\alpha,\alpha)\}.\]
A mapping $T:\Sigma_\alpha \to{\mathcal L}(X)$ is called a \emph{semigroup} if $T(z_1+z_2)=T(z_1)T(z_2)$ for all $z_1,z_2\in \Sigma_\alpha$. We speak of a \emph{holomorphic semigroup} if in addition the mapping $T$ is holomorphic. Then the restriction of $T$ to $(0,\infty)$ is a strongly continuous semigroup in the sense of Definition~\ref{def:5.1}. 
\begin{thm}\label{th:7.1}
Let $\alpha\in (0,\frac{\pi}{2}]$, $M>0$. Let $T:\Sigma_\alpha \to {\mathcal L}(X)$ be a holomorphic semigroup such that 
\[  \|T(z)\|\leq M \mbox{ for all }z\in \Sigma_\alpha. \]
Then the generator $A$ of $T$ (which is a relation) satisfies the following holomorphic estimate:
\begin{equation}\label{eq:7.1} 
\Sigma_{\alpha +\frac{\pi}{2}}\subset \rho (A)\mbox{ and }\|  \lambda R(\lambda ,A)\|\leq M\frac{1}{\sin \varepsilon}\mbox{ for all }\lambda\in \Sigma_{\alpha +\frac{\pi}{2} -\varepsilon}\mbox{ and }0<\varepsilon<\frac{\pi}{2}.
\end{equation}	
\end{thm}   
  \begin{proof}
	It follows from \cite[Theorem 2.6.1]{ABHN11} and its proof that the resolvent of $A$ has a holomorphic extension to $\Sigma_{\alpha +\frac{\pi}{2}}$ satisfying the estimate \eqref{eq:7.1}. Lemma~\ref{lem:3.6new} shows that the holomorphic extension is actually the resolvent of $A$ on $\Sigma_{\alpha+\frac{\pi}{2}}$.   
\end{proof}	
\begin{cor}\label{cor:7.2}
In the situation of Theorem~\ref{th:7.1}, for each $x\in X$, $u(t):=T(t)x$ defines a function $u\in C^\infty ((0,\infty);X)$ such that 
\begin{equation}\label{eq:7.2}
\dot{u}(t)\in Au(t)\mbox{ for all }t>0. 
\end{equation} 
\end{cor}
\begin{proof}
Let $x\in X$, $u(t)=T(t) x$. Then $u\in C^\infty((0,\infty);X)$. By the definition of the generator, \[R(\lambda,A)x=\int_0^\infty e^{-\lambda t} u(t) dt =\int_0^\infty \lambda e^{-\lambda t} v(t)dt\]
where $v(t)=\int_0^\infty  u(s)ds$. Since $(R(\lambda,A)x,\lambda R(\lambda,A)x-x)\in A$, it follows that 
\[  \int_0^\infty \lambda e^{-\lambda t} (v(t),u(t)-x)dt \in A\mbox{ for all }\lambda >0. \]
By Theorem~\ref{th:5.7} this implies that $(v(t),u(t)-x)\in A$ for all $t>0$. Since $A$ is closed, also the derivative $(u(t),\dot{u}(t))\in A$.
\end{proof}
The converse of Theorem~\ref{th:7.1} has the following form. 
\begin{thm}\label{th:7.3}
Let $\alpha\in [0,\frac{\pi}{2}]$, $M>0$ and let $A\subset X\times X$ be a closed relation such that 
\begin{equation}\label{eq:7.3}
\Sigma_{\alpha +\frac{\pi}{2}}\subset \rho(A)\mbox{ and }\|\lambda R(\lambda,A )\|\leq M\mbox{ for all }\lambda\in \Sigma_{\alpha +\frac{\pi}{2}}.	
\end{equation}	 
Then $A$ generates a holomorphic semigroup $T:\Sigma_\alpha\to{\mathcal L}(X)$ satisfying 
\begin{equation}\label{eq:7.4}
\|T(z)\|\leq M\left( 1+\frac{2e\pi}{\sin \varepsilon}	\right)\mbox{ for all }z\in \Sigma_{\alpha -\varepsilon},\, \, 0<\varepsilon<\alpha.
\end{equation}	
\end{thm}
\begin{proof}
	The proof of \cite[Theorem 2.6.1]{ABHN11} yields a holomorphic function $T:\Sigma_\alpha \to {\mathcal L}(X)$ such that 
	\eqref{eq:7.4} holds and \[R(\lambda ,A)=\int_0^\infty e^{-\lambda t}T(t) dt\mbox{ for all } \lambda >0.\]
	Now the proof of \cite[Theorem 3.1.7]{ABHN11} shows that $T(s+t)=T(s)T(t)$ for all $t,s>0$. A standard argument involving holomorphy (see \cite[Proposition 3.7.2(a)]{ABHN11}) shows that the semigroup property also holds on $\Sigma_\alpha$.      
\end{proof}
Next we establish a convergence result. It is one of our main results. Since we admit relations, in contrast to the usual convergence results, no assumption on the limit is required.
\begin{thm}\label{th:7.4}
	Let $0<\alpha\leq \frac{\pi}{2}$, $M>0$ and let $T_n:\Sigma_\alpha \to{\mathcal L}(X)$ be holomorphic semigroups satisfying 
	\[  \|T_n (z)\|\leq M\mbox{ for all }n\in \N,\,\,z\in \Sigma_\alpha.\]
	Denote by $A_n$ the generator of $T_n$ and assume that there exists $\lambda_0\in \Sigma_{\alpha+\frac{\pi}{2}}$ such that $(R(\lambda_0,A_n))_{n\in\N}$ converges strongly and $A=\lim_{n\to\infty}A_n $. Then $A$ generates a bounded holomorphic semigroup $T:\Sigma_\alpha\to{\mathcal L}(X)$. Moreover, for all $x\in X$, 
	\[\lim_{n\in\infty} T_n (z)x=T(z)x\mbox{ uniformly for }z\in K,  \]
	whenever $K\subset \Sigma_\alpha$ is compact.     
\end{thm} 
\begin{proof}
	Let $U=\rho(A)\cap \Sigma_{\alpha+\frac{\pi}{2}}$. Then $U$ is open and $\lambda_0\in U$ by Theorem~\ref{th:3.5}. We show that $U$ is relatively closed in $\Sigma_{\alpha+\frac{\pi}{2}}$. To that aim, let $\lambda_k\in U$ and $\lambda_\infty\in \Sigma_{\alpha +\frac{\pi}{2}}$ such that $\lim_{k\to\infty} \lambda_k=\lambda_\infty$. Then there exists $0<\varepsilon<\alpha$ such that $\lambda_k,\lambda_\infty\in \Sigma_{\alpha+\frac{\pi}{2}}$ for all $k\in\N$. Thus 
	\[  \|\lambda_k R(\lambda_k,A_n)\|\leq \frac{M}{\sin \varepsilon}\]  
	for all $k\in\N$, $n\in\N$. It follows from Theorem~\ref{th:3.5}  that $R(\lambda_k,A)=\lim_{n\to\infty}R(\lambda_k,A_n)$ strongly. Hence $\|\lambda_k R(\lambda_k,A)\|\leq \frac{M}{\sin\varepsilon}$ for all $k\in \N$.  Now Corollary~\ref{cor:3.2} implies that $\lambda_\infty\in \rho(A)$. This proves the claim.  Since $\Sigma_{\alpha+\frac{\pi}{2}}$ is connected we deduce that $U=\Sigma_{\alpha+\frac{\pi}{2}}$; i.e. $ \Sigma_{\alpha+\frac{\pi}{2}}\subset \rho(A)$.  It follows from Theorem~\ref{th:3.5} that $\lim_{n\to\infty} R(\lambda, A_n)=R(\lambda, A)$ strongly for all $\lambda\in \Sigma_{\alpha+\frac{\pi}{2}}$.  In order to prove strong convergence of the semigroups we consider the Banach space \[\ell^\infty(X):=\{  x=(x_n)_{n\in\N}  :\|x\|_\infty:=\sup_{n\in\N} \|x_n\|<\infty\}\]
	and its closed subspace $c(X)$ of all convergent sequences.  Let $x\in X$. Define $F:\Sigma_\alpha\to\ell^\infty(X)$ by 
	$F(z)=(T_n(z)x)_{n\in\N}$. Then $F$ is bounded and holomorphic by \cite[Theorem A.7]{ABHN11}. Let $F_0$ be the restriction of $F$ to $(0,\infty)$. The Laplace transform $\widehat{F_0}$ of $F_0$ is given by 
$\widehat{F_0}(\lambda)=(R(\lambda,A_n)x)_{n\in\N}\in c(X)$ for all $\lambda\in \Sigma_{\alpha+\frac{\pi}{2}}$. It follows from  Theorem~\ref{th:5.7}  that $F(t)\in c(X)$ for all $t>0$;  i.e. $T(t)x:=\lim_{n\to\infty}T_n(t)x$ exists for all $x\in X$. Now it follows from Vitali's Theorem \cite[Theorem A.5]{ABHN11} that $\lim_{n\to\infty}T_n (z)x=:T(z)x$ exists for all $x\in X$ and all $z\in \Sigma_\alpha$, uniformly with respect to $z$ on all compact subsets of $\Sigma_\alpha$. Thus $T:\Sigma_\alpha\to{\mathcal L}(X)$ is a bounded holomorphic semigroup. Since for $x\in X$, $\lambda>0$, 
\begin{eqnarray*}
	 \widehat{T}(\lambda )x & = & \int_0^\infty e^{-\lambda t} T(t)x\, dt \\
	  & = & \lim_{n\to\infty}\int_0^\infty e^{-\lambda t} T_n(t)x\, dt =\lim_{n\to \infty}R(\lambda , A_n) x \\
 & = & R(\lambda, A)x,
 \end{eqnarray*}	  
the generator of $T$ is $A$. 
\end{proof}
\section{Convergence of heat semigroups}\label{sec:8}
The convergence results of Section~\ref{sec:6} had the advantage of being elementary. To show that the relation $A_\Omega$ is $m$-dissipative was based on our version of Theorem~\ref{th:1.2}  of the Lumer--Phillips Theorem. However, the convergence results for holomorphic semigroups developed in the preceding  section give better convergence properties.  

We start showing that the relation $A_\Omega$ generates a holomorphic semigroup.  Let $B\subset \R^d$, $d\geq 2$ be a large ball and let 
\[   W=\{ \Omega\subset B:\Omega\mbox{ open and Dirichlet regular}\}.\]
For $\Omega\in W$ we consider the $m$-disspative relation $A_\Omega\subset C(\overline{B})\times C (\overline{B})$ from Section~\ref{sec:6}. Then the following generation theorem holds. 
\begin{thm}\label{th:8.1}
	There exist $\alpha\in (0,\frac{\pi}{2}]$ and $M>0$ such that for each $\Omega\in W$ the relation $A_\Omega$ generates a holomorphic semigroup \[T_\Omega:\Sigma_\alpha\to {\mathcal L}(C(\overline{B}))\] satisfying $\|T_\Omega(z)\|\leq M$ for all $z\in \Sigma_\alpha$.  
\end{thm}
\begin{proof}
By the proof of \cite[Theorem 6.1.9]{ABHN11} for each $\Omega\in W$ there exists $c_\Omega\geq 0$ such that 
\[   \|\lambda R(\lambda, A_\Omega)\|\leq c_\Omega M\mbox{ for all }\lambda\in\C \mbox{ with }\re (\lambda)>0,\]
where $M$ is independent of $\Omega$ and $c_\Omega=1+\|B_\Omega^{-1}\|$, where $B_\Omega$ is the \emph{Poisson operator} defined on $Y:=C(\overline{\Omega})\times C(\partial \Omega)$ by 
\[  D (B_\Omega )=\{ (u,0)\in Y:\Delta u\in C(\overline{\Omega})\}\mbox{ and }B_\Omega (u,0)=(\Delta u,-u_{|\partial \Omega}).\]
In the proof of \cite[Theorem 6.1.9]{ABHN11} it is shown that $B_\Omega$ is bijective and for $(f,\varphi)\in Y$, $(-B_\Omega)^{-1}(f,\varphi)=:u=v+w$ where $w=-E_d\star \tilde{f}$, $\tilde{f}$ the extension of $f$ by $0$, $E_d$ the Newtonian potential, $v\in C(\overline{\Omega})$ harmonic on $\Omega$ such that $v_{|\partial \Omega}=\varphi-w_{|\partial \Omega}$. Thus, by the maximum principle, 
\[  \|u\|_{C(\overline{\Omega})}\leq \|v\|_{C(\overline{\Omega})} + \|w\|_{C(\overline{\Omega})}\leq 
\|\varphi\|_{C(\partial\Omega )}   +2\|w\|_{C(\overline{\Omega})} .    \]  
But 
\[ \|w\|_{C(\overline{\Omega})} =\|E_d\star \tilde{f}\|_{C(\overline{\Omega})}\leq \|E_d\|_{L^1(B+B)}\|f\|_{C(\overline{B})}, \]
a bound which is independent of $\Omega$. 
Thus \[\|\lambda R(\lambda, A_\Omega)\|\leq cM   \mbox{ for all }\lambda\in\C\mbox{ with }\re (\lambda) >0,\]
for some $c>0$ and all $\Omega\in W$. 

Now the power series argument of \cite[Corollary 3.7.12]{ABHN11} yields $\beta\in (0,\frac{\pi}{2}]$, $\tilde{M}>0$ such that 
\[  \Sigma_{\beta +\frac{\pi}{2}}\subset \rho(A_\Omega)\mbox{ and }\|\lambda R(\lambda,A_\Omega)\|\leq \tilde{M}\mbox{ for all }\lambda\in \Sigma_{\beta+\frac{\pi}{2}}\mbox{ and all }\Omega\in W.\]
Theorem~\ref{th:7.3} gives the desired result. 
 \end{proof}
Let $\Omega\in W$. Recall that 
\[C_0(\Omega)= \{   v\in C(\overline{B}):v=0\mbox{ on }\overline{B}\setminus \Omega\}.\]
The semigroup $T_\Omega$ can be seen as a function in $C^\infty((0,\infty);{\mathcal L}(C(\overline{B})))$. We now characterize the orbits $T_\Omega (\cdot)u_0$ as solutions of the heat equation in the following way. 
\begin{thm}\label{th:8.2}
The function $u=T_\Omega ( \cdot{} )u_0$ is the unique solution of 
\begin{equation}\label{eq:8.1}
 \dot{u}(t)=\Delta u(t) \mbox{ in }{\mathcal D}(\Omega)'\mbox { for all }t>0
 \end{equation}
with $u\in C^\infty((0,\infty);C_0(\Omega))$ and $\lim_{t\to 0+}u(t)=u_0\mbox{ in }L^2(\Omega)$.
\end{thm}   
\begin{proof}
Let $g\in C(\overline{B})$, $u(t)=T_\Omega (t)g$. By Corollary~\ref{cor:7.2}, we have  $u\in C^\infty((0,\infty);C(\overline{B}))$ and $\dot{u}(t)\in A_\Omega u(t)$ for all $t>0$. In particular $u(t)\in \dom A_\Omega\in C_0(\Omega)$ and $\dot{u}(t)=\Delta u(t)$ in ${\mathcal D}(\Omega)'$. In order to show that $\lim_{t\to 0+} u(t)_{|\Omega}=g_{|\Omega}$ in $L^2(\Omega),$ we consider  the Dirichlet Laplacian $\Delta_\Omega$ on $L^2(\Omega)$ given by 
\[  \dom \Delta_\Omega :=\{ v\in H_0^1(\Omega):\Delta v\in L^2(\Omega)\}\mbox{ and } \Delta_\Omega v:=\Delta v.\]
Then $\Delta_\Omega$ generates a $C_0$-semigroup $T_2$ on $L^2(\Omega)$. We claim that for $g\in C(\overline{B})$, $\lambda>0$, $(R(\lambda, A_\Omega)g)_{|\Omega}=R(\lambda,\Delta_\Omega)(g_{|\Omega})$. In fact, let $v=R(\lambda, A_\Omega)g$. Then $v\in \dom A_\Omega \subset C_0(\Omega)$ and $\lambda v-g \in A_\Omega v$. Hence $\lambda v-g=\Delta v$ in ${\mathcal D}(\Omega)'$. It follows from \cite[Theorem 2.5]{AD07} or \cite[Lemma 4.2]{AB99} that $v_{|\Omega}\in H_0^1(\Omega)$. As a consequence, $v_{|\Omega} \in D(\Delta_\Omega)$ and $\lambda v_{|\Omega}-\Delta v_{|\Omega}=g_{|\Omega}$. This is the claim. It follows from the uniqueness theorem \cite[Theorem 1.7.3]{ABHN11} that $(T_\Omega (t)g)_{|\Omega}=T_2(t)(g_{|\Omega})$. 
Hence $\lim_{t\to 0+} u(t)_{|\Omega}=g_{|\Omega}$ in $L^2(\Omega)$. 

In oder to show uniqueness, let $u$ be a solution of \eqref{eq:8.1}. Since for $t>0$, $u(t)\in C_0(\Omega)$ and 
$\Delta u(t)_{||\Omega}=\dot{u}(t)_{|\Omega}\in C(\overline{\Omega})$, it follows as before that  $u(t)_{|\Omega}\in H_0^1(\Omega)$. Thus $\dot{u}(t)_{|\Omega}=\Delta_\Omega u(t)$. Since $\lim_{t\to 0+} u(t)_{|\Omega} =g_{|\Omega}$ in $L^2(\Omega)$, it follows that $u(t)_{|\Omega}=T_2(t)(g_{|\Omega})$. This proves uniqueness since $u(t)=0$ outside $\Omega$.     	
\end{proof}
The mapping $\Omega\mapsto T_\Omega$ is continuous in the following sense. We keep the notations of Theorem~\ref{th:8.1}. 
\begin{thm}\label{th:8.3}
Let $\Omega_n,\Omega\in W$ such that $\Omega_n\to\Omega$ as $n\to\infty$. Assume that $\Omega$ is stable. Then for each $g\in C(\overline{\Omega})$, 
\[  T_n(z)g\to T(z)g\mbox{ in }C(\overline{B})\mbox{ as }n\to\infty,\]
uniformly with respect to  $z$ on compact subsets of $\Sigma_\alpha$.   
\end{thm}
\begin{proof}	
By Theorem~\ref{th:6.3}, for $\lambda >0$, $\lim_{n\to\infty} R(\lambda, \Omega_n)=R(\lambda, \Omega)$ strongly.  Thus, in view of Theorem~\ref{th:8.1} the claim follows from Theorem~\ref{th:7.4} 	
\end{proof}
Theorem~\ref{th:8.3} can be rewritten in terms of the solutions of the heat equation.  The point is that the convergence holds in $C(\overline{B})$ for the uniform norm (and not merely in $L^2(\Omega)$ as in \cite[Theorem 6.2]{taiwan}).   
\begin{cor}\label{cor:8.4}
Let $\Omega_n,\Omega_\infty\in W$ such that $\Omega_n\to \Omega_\infty$ as $n\to\infty$. Assume that $\Omega_\infty$ is stable. Let $u_0\in C(\overline{B})$ and let $u_n$ be the solution of 
 	\begin{equation}\label{eq:8.2}
 		\dot{u_n}(t)=\Delta u_n(t) \mbox{ in }{\mathcal D}(\Omega_n)'\mbox { for all }t>0
 	\end{equation}
 	with $u_n\in C^\infty((0,\infty);C_0(\Omega_n))$, $\lim_{t\to 0+}u_n(t)=u_0\mbox{ in }L^2(\Omega)$ and $n\in\N\cup\{\infty\}$. 
 	
 	Then $u_n(t)\to u_\infty(t)$ in $C(\overline{B})$ uniformly on $[\frac{1}{\tau},\tau]$ for all $\tau>1$. 
\end{cor}
Finally we mention that another mode of convergence of $\Omega_n$ to $\Omega$ is studied in \cite[Definition 5.1 and Theorem 5.5]{AD07}. Theorem~\ref{th:8.3} remains true for this alternative mode of convergence.   \\

\noindent \textbf{Acknowledgments:} The authors are grateful to the University of the Witswatersrand for its financial support which made this scientific collaboration possible.     

\nocite{RDV3}
\bibliographystyle{abbrv}

\end{document}